\documentclass[11 pt,final]{article}

\usepackage{amsmath,amssymb,amsthm,amscd,amsfonts}
\usepackage{epsfig,pinlabel}
\usepackage[hmargin=1.25in, vmargin=1in]{geometry}
\usepackage[font=small,format=plain,labelfont=bf,up,textfont=it,up]{caption}
\usepackage{subcaption}
\usepackage{stmaryrd}
\usepackage{verbatim}
\usepackage{type1cm}
\usepackage{url}
\usepackage{calc}
\usepackage[all,cmtip]{xy}
\usepackage{graphicx}
\usepackage{multicol}
\usepackage{pstricks}
\usepackage{relsize}
\usepackage{float}

\newcommand{\urlwofont}[1]{\urlstyle{same}\url{#1}}

\newcommand{\nc}{\newcommand}
\nc{\nt}{\newtheorem}
\nc{\dmo}{\DeclareMathOperator}

\theoremstyle{plain}
\newtheorem{theorem}{Theorem}[section]
\newtheorem{maintheorem}{Theorem}
\newtheorem{prop}[theorem]{Proposition}
\newtheorem{lemma}[theorem]{Lemma}

\newtheorem{corollary}[theorem]{Corollary}

\theoremstyle{definition}

\theoremstyle{remark}

\DeclareMathOperator{\Mod}{Mod}

\dmo{\SMod}{SMod}
\dmo{\PMod}{PMod}
\dmo{\SHomeo}{SHomeo}
\dmo{\SI}{\mathcal{SI}}
\dmo{\SSp}{SSp}
\dmo{\PSp}{PSp}
\DeclareMathOperator{\IA}{IA}

\newcommand\Z{\ensuremath{\mathbb{Z}}}

\nc{\p}[1]{\noindent {\bf #1.}}
\nc{\margin}[1]{\marginpar{\scriptsize #1}}

\nc{\PartialIBases}{\mathfrak{IB}}
\nc{\PartialIBasesEx}{\widehat{\mathfrak{IB}}}
\nc{\PartialBases}{\mathfrak{B}}
\nc{\Building}{\mathfrak{T}}
\nc{\height}{\ensuremath{\text{ht}}}
\nc{\Poset}{\mathfrak{P}}
\nc{\Field}{\mathbb{F}}
\nc{\Link}{\ensuremath{\text{Link}}}
\nc{\Star}{\ensuremath{\text{Star}}}
\nc{\Aut}{\ensuremath{\text{Aut}}}

\nc{\SymTorelli}{\ensuremath{\mathcal{SI}}}
\nc{\BTorelli}{\ensuremath{\mathcal{BI}}}
\dmo{\Braid}{\ensuremath{B}}
\dmo{\PureBraid}{\ensuremath{PB}}
\nc{\Hyper}{\ensuremath{\iota}}

\nc{\BigFreeProd}{\mathop{\mbox{\Huge{$\ast$}}}}

\nc{\Quotient}{\ensuremath{\mathcal{Q}}}
\nc{\QuotientEx}{\ensuremath{\widehat{\mathcal{Q}}}}

\nc{\Presentation}[2]{\ensuremath{\text{$\langle #1$ $|$ $#2 \rangle$}}}
\nc{\SpGen}{\ensuremath{S_{\text{Sp}}}}
\nc{\SpRel}{\ensuremath{R_{\text{Sp}}}}
\nc{\QGen}{\ensuremath{S_{\mathcal{Q}}}}
\nc{\QRel}{\ensuremath{R_{\mathcal{Q}}}}
\nc{\PBs}{\ensuremath{T}}
\nc{\Qs}{\ensuremath{\overline{s}}}

\dmo{\PB}{PB}
\nc{\BIredg}{\mathcal{BI}_{2g+1}^{\text{red}}}
\nc{\BI}{\mathcal{BI}}
\dmo{\D}{D}
\dmo{\Stab}{Stab}
\dmo{\Surger}{Surger}
\nc{\I}{\mathcal{I}}

\nc{\spanmap}{span}

\nc{\genbygen}[2]{\premonoid{#1}{#2}}
\nc{\premonoid}[2]{#1 \circledcirc #2}
\nc{\monoid}[2]{#1 \odot #2}

\nc{\raag}{A_\Gamma}
\nc{\raagdelt}{A_\Delta}
\nc{\autraag}{\mathrm{Aut}(\raag)}
\nc{\outraag}{\mathrm{Out}(A_\Gamma)}
\nc{\autraagdelt}{\mathrm{Aut}(A_\Delta)}
\nc{\glk}{\mathrm{GL}(k,\mathbb{Z})}
\nc{\gln}{\mathrm{GL}(n,\mathbb{Z})}
\nc{\glkdelt}{\mathrm{GL}(k |\Delta| ,\mathbb{Z})}
\nc{\gldelt}{\mathrm{GL}(|\Delta| ,\mathbb{Z})}
\nc{\zkdelt}{\mathbb{Z}^{k|\Delta|}}
\nc{\aut}{\mathrm{Aut}}
\nc{\out}{\mathrm{Out}}
\nc{\join}{\mathcal{J}}
\nc{\pc}{\mathrm{PC}}
\nc{\lk}{\mathrm{lk}}
\nc{\st}{\mathrm{st}}
\nc{\inn}{\mathrm{Inn}}

\nc{\pia}{\Pi \mathrm{A}}
\nc{\ppia}{\mathrm{P \Pi A}}
\nc{\epia}{\mathrm{E \Pi A}}
\nc{\ptor}{\mathcal{PI}}
\nc{\pbc}{\mathfrak{B}^\pi}
\nc{\cay}{\mathrm{Cay}}
\nc{\rev}{\mathrm{rev}}
\nc{\autfn}{\mathrm{Aut}(F_n)}

\include{diagram}
\title{\vspace{-30pt} A generating set for the palindromic Torelli group}

\author{Neil J. Fullarton\vspace{-6pt}}

\begin{document}

\newcounter{enumi_saved}

\maketitle

\vspace{-23pt}
\begin{abstract}A \emph{palindrome} in a free group $F_n$ is a word on some fixed free basis of $F_n$ that reads the same backwards as forwards. The \emph{palindromic automorphism group} $\pia_n$ of the free group $F_n$ consists of automorphisms that take each member of some fixed free basis of $F_n$ to a palindrome; the group $\pia_n$ has close connections with hyperelliptic mapping class groups, braid groups, congruence subgroups of $\gln$, and symmetric automorphisms of free groups. We obtain a generating set for the subgroup of $\pia_n$ consisting of those elements that act trivially on the abelianisation of $F_n$, the \emph{palindromic Torelli group} $\ptor_n$. The group $\ptor_n$ is a free group analogue of the hyperelliptic Torelli subgroup of the mapping class group of an oriented surface. We obtain our generating set by constructing a simplicial complex on which $\ptor_n$ acts in a nice manner, adapting a proof of Day--Putman \cite{DayPutman}. The generating set leads to a finite presentation of the principal level 2 congruence subgroup of $\gln$. \end{abstract}

\section{Introduction} Let $F_n$ be the free group of rank $n$ on some fixed free basis $X$. The palindromic automorphism group of $F_n$, denoted $\pia_n$, consists of automorphisms of $F_n$ that take each member of $X$ to some palindrome, that is, a word on $X$ that reads the same backwards as forwards. Collins \cite{Collins} introduced the group $\pia_n$ and proved that it is finitely presented, giving an explicit presentation. Glover--Jensen \cite{GloverJensen} obtained further results about $\pia_n$, utilising a contractible subspace of the auter space of $F_n$ on which $\pia_n$ acts cocompactly, with finite stabilisers. For instance, they calculate that the virtual cohomological dimension of $\pia_n$ is $n-1$. The group $\pia_n$ is a free group analogue of the hyperelliptic mapping class group of an oriented surface; we develop this analogy later in this introduction.

In this paper, we are primarily concerned with the intersection of $\pia_n$ with the Torelli subgroup of $F_n$, that is, the subgroup of automorphisms of $\pia_n$ that act trivially on the abelianisation of $F_n$. We denote this intersection by $\ptor_n$, and refer to it as the \emph{palindromic Torelli group} of $F_n$. Little appears to be known about the group $\ptor_n$: Collins \cite{Collins} first observed that it is non-trivial, and Jensen--McCammond--Meier \cite[Corollary 6.3]{JensenMcCammondMeier} showed that $\ptor_n$ is not of finite homological type for $n \geq 3$. In Section 2, we introduce non-trivial members of $\ptor_n$ ($n \geq 3$) known as \emph{doubled commutator transvections} and \emph{separating $\pi$-twists}. The main theorem of this paper establishes that these generate $\ptor_n$.

\begin{maintheorem}\label{mainthm}The group $\ptor_n$ ($n \geq 3$) is generated by doubled commutator transvections and separating $\pi$-twists.
\end{maintheorem}

In order to prove Theorem~\ref{mainthm}, we establish finite generating sets for the subgroups of $\pia_n$ consisting of automorphisms that fix each member of some specified subset of the free basis $X$. These generating sets, which are given precisely in the statement of Proposition~\ref{palgens}, are obtained by utilising Stallings' graph folding algorithm.

Let $\Gamma_n[2]$ denote the \emph{principal level 2 congruence subgroup of $\gln$}, that is, is the kernel of the surjection $\gln \to \mathrm{GL}(n, \Z/2)$ that reduces matrix entries mod 2. In Section 2, we discuss a short exact sequence with kernel the palindromic Torelli group and quotient $\Gamma_n[2]$. For $1 \leq i \neq j \leq n$, let $S_{ij} \in \Gamma_n[2]$ be the matrix that has 1s on the diagonal and 2 in the $(i,j)$ position, with 0s elsewhere, and let $O_i \in \Gamma_n[2]$ differ from the identity only in having $-1$ in the $(i,i)$ position. The following corollary of Theorem \ref{mainthm} provides a finite presentation of $\Gamma_n[2]$ for $n \geq 4$. 

\begin{corollary}\label{congpres}The principal level 2 congruence group $\Gamma_n[2]$ ($n \geq 4$) is generated by $$\{ S_{ij}, O_i \mid 1 \leq i \neq j \leq n \},$$ subject to the defining relators

\begin{multicols}{2}
\begin{enumerate}
    \item ${O_i}^2$
    \item $[O_i,O_j]$
    \item $(O_iS_{ij})^2$
    \item $(O_jS_{ij})^2$
    \item $[O_i, S_{jk}]$
    \item $[S_{ki}, S_{kj}]$
    \item $[S_{ij},S_{kl}]$
    \item $[S_{ji},S_{ki}]$
    \item $[S_{kj},S_{ji}]{S_{ki}}^{-2}$
    \item $(S_{ij}{S_{ik}}^{-1}S_{ki}S_{ji}S_{jk}{S_{kj}}^{-1})^2$
\end{enumerate}
\end{multicols} where $1 \leq i,j,k,l \leq n$ are pairwise different.
  \end{corollary}
We note that in the proof of Theorem \ref{mainthm} it becomes apparent that not every relator of type 10 is needed. In fact, for each choice of three indices $i$, $j$ and $k$, we need only select one such word (and disregard the others, for which the indices have been permuted).

We also derive the following similar presentation for $\Gamma_n[2]$ when $n = 2$ or $3$, however these are acquired independently of Theorem~\ref{mainthm}. Indeed, the presentation of $\Gamma_3[2]$ is used to obtain a generating set for $\ptor_3$, which forms the base case of an inductive proof of Theorem~\ref{mainthm}.

\begin{prop}\label{congpres3}The principal level 2 congruence group $\Gamma_n[2]$ ($n= 2,3$) is generated by $$\{ S_{ij}, O_i \mid 1 \leq i \neq j \leq n \},$$ subject to the defining relators in the statement of Corollary~\ref{congpres} of types:
\begin{itemize} \item (1)--(4) for $n=2$,
\item (1)--(6), (8)--(10) for $n=3$. \end{itemize}
\end{prop}

A key tool in the proof of Proposition~\ref{congpres3} is an `even' version of the Division Algorithm for the integers. This is the observation that under certain circumstances, the quotient $q \in \Z$ given when dividing $a \in \Z$ by $b \in \Z$ may be chosen to be even, if we sacrifice control of the sign of the remainder $r \in \Z$. More details of this procedure are given in the proofs of Lemma~\ref{conggens} and Theorem~\ref{augment}.

A similar presentation for $\Gamma_n[2]$ was recently found independently by Kobayashi \cite{Kobayashi}, and was also known to Margalit--Putman \cite{Margalit}. As pointed out by Margalit--Putman, this is a natural presentation for $\Gamma_n[2]$, as relators of types (6)--(9) bear a strong resemblance to the Steinberg relations that hold between the transvections generating $\mathrm{SL}(n, \Z)$ \cite[\S 5]{Milnor}.

\textbf{A comparison with mapping class groups.} While $\pia_n$ is defined entirely algebraically, it may viewed as a free group analogue of a subgroup of the mapping class group of an oriented surface. Let $S_g$ and $S_g^1$ denote the compact, connected, oriented surfaces of genus $g$ with 0 and 1 boundary components, respectively. We shall use $S$ to denote such a surface, with or without boundary. Recall that the \emph{mapping class group} of the surface $S$, denoted $\Mod (S)$, consists of isotopy classes of orientation-preserving self-homeomorphisms of $S$, where isotopies are required to fix any boundary component pointwise at all times. For a self-homeomorphism $f$ of $S$, we denote its isotopy class by $[f]$.

A \emph{hyperelliptic involution} of the surface $S$ is an order 2 homeomorphism of the surface that acts as $-I$ on $H_1(S, \Z)$ \cite[Sections 2 \& 4]{BrendleMargalit}. Let $s$ denote the homeomorphism of $S_g^1$ seen in Figure \ref{1bdry}. By capping the boundary with a disk, the map $s$ induces a homeomorphism of $S_g$, which we also denote $s$, by an abuse of notation. The map $s$ is an example of a hyperelliptic involution of $S_g^1$ (and $S_g$). We note that the mapping class of any hyperelliptic involution in $\Mod(S_g)$ ($g \geq 1$) is conjugate to $[s]$ \cite[Proposition 7.15]{primer}.
\begin{figure}
\centering
\begin{subfigure}{.5\textwidth}
\centering
\def\svgwidth{2.6in}
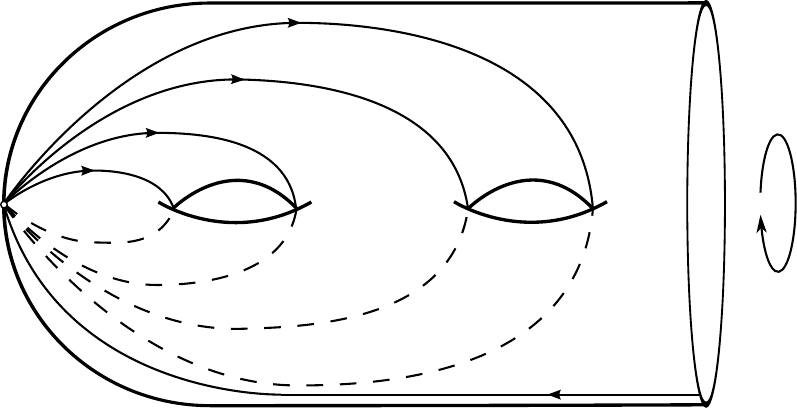
  \caption{}
    \label{1bdry}
\end{subfigure}%
\begin{subfigure}{.5\textwidth}
\centering
\def\svgwidth{2.5in}
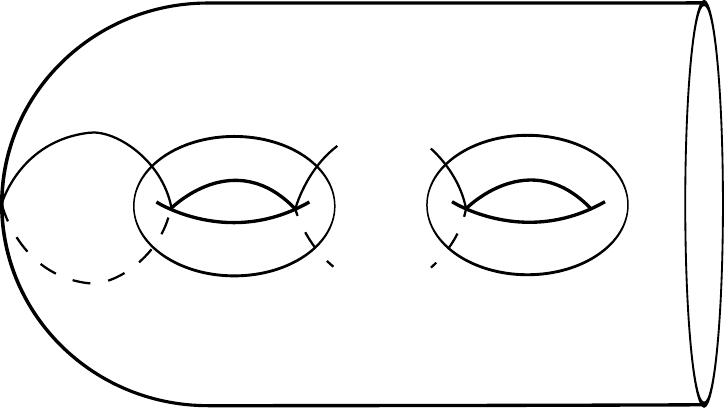
  \caption{}
    \label{chain}
\end{subfigure}
\caption{\\
\textbf{(a)} The involution $s$ rotates the surface by $\pi$ radians. Under the classical Nielsen embedding, we may view the braid group $B_{2g} \leq \SMod(S_g^1)$ as a subgroup of $\pia_{2g} \leq \aut(F_{2g})$, where $F_{2g}$ is the free group on the oriented loops $x_1, \ldots, x_{2g}$.
\\
\textbf{(b)} The standard symmetric chain in $S_g^1$. The Dehn twists about $c_1, \ldots, c_{2g}$ generate $\SMod(S_g^1) \cong B_{2g+1}$.}
\end{figure}

The \emph{hyperelliptic mapping class group of the surface $S_g$}, denoted $\SMod(S_g)$, is the centraliser of $[s]$ in $\Mod(S_g)$. Although $[s] \not \in \Mod(S_g^1)$, as $s$ does not fix the boundary of $S_g^1$, we define the \emph{hyperelliptic mapping class group of $S_g^1$}, denoted $\SMod(S_g^1)$, to be the group of isotopy classes of the centraliser of $s$ in $\mathrm{Homeo}^+(S_g^1)$ \cite[Chapter 9]{primer}.

An obvious analogue of a hyperelliptic involution in $\autfn$ is an order 2 member of $\autfn$ that acts as $-I$ on $H_1(F_n, \Z) = \Z^n$. An example of such an involution in $\autfn$ is the automorphism $\iota$ that inverts each member of the free basis $X$. An analogy between $s$ and $\iota$ is strengthened by two observations. Firstly, Glover--Jensen \cite[Proposition 2.4]{GloverJensen} showed that any hyperelliptic involution in $\autfn$ is conjugate to $\iota$. Secondly, the action of $s$ on $\pi_1(S_g^1) = F_{2g}$, with free basis as seen in Figure \ref{1bdry}, is to invert each member of the free basis, as $\iota$ does. It is easily verified that $\pia_n$ is the centraliser of $\iota$ in $\autfn$ \cite[Section 2]{GloverJensen}, so we may think of $\pia_n$ as being a free group analogue of the hyperelliptic mapping class groups $\SMod(S_g)$ and $\SMod(S_g^1)$.

The comparison between $\pia_n$ and $\SMod(S_g^1)$ is made more precise using the classical Nielsen embedding $\Mod(S_g^1) \hookrightarrow \aut(F_{2g})$. Take the $2g$ oriented loops seen in Figure \ref{1bdry} as a free basis for $\pi_1(S_g^1)$. Observe that $s$ acts on these loops by switching their orientations. In order to use Nielsen's embedding into $\aut(F_{2g})$, we must take these loops to be based on the boundary; we surger using the arc $\mathcal{A}$ to achieve this. The group $\SMod(S_g^1)$ is isomorphic to the braid group $B_{2g+1}$ by the Birman--Hilden theorem \cite{BirmanHilden}, and is generated by Dehn twists about the curves in the standard, symmetric chain on $S_g^1$, seen in Figure \ref{chain}. The Dehn twists about the $2g-1$ curves $c_2, \ldots, c_{2g}$ generate the braid group $B_{2g}$. Taking the loops seen in Figure \ref{1bdry} as our free basis $X$, a straightforward calculation shows that the images of these $2g-1$ twists in $\aut(F_{2g})$ lie in $\pia_{2g}$. Specifically, the twist about $c_{i+1}$ is taken to the automorphism $Q_i$ of the form
\begin{eqnarray*}
 x_i & \mapsto & x_{i+1}, \\
 x_{i+1} & \mapsto & x_{i+1} {x_i}^{-1} x_{i+1}, \\
 x_j &\mapsto & x_j
\end{eqnarray*} for $1 \leq i < 2g$ and $j \neq i, i+1$.
This shows that $\pia_n$ contains the braid group $B_n$ as a subgroup, when $n$ is even. This embedding of $B_n$ is a restriction of one studied by Perron--Vannier \cite{PerronVannier} and Crisp--Paris \cite{CrispParis}. When $n$ is odd, we also have $B_n \hookrightarrow \pia_n$, since discarding $Q_1$ gives a generating set for $B_{2g-1}$ inside $\pia_{2g-1} \leq \aut(F_{2g})$.

\begin{figure}
\centering
\def\svgwidth{2.5in}
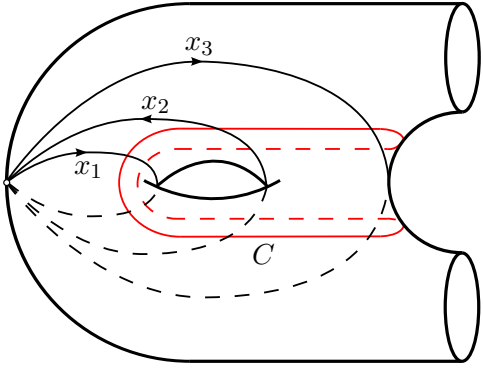
\caption{The Dehn twist about the symmetric, separating curve $C$ maps to a separating $\pi$-twist in $\ptor_{2g}$ under the Nielsen embedding. Note that we only depict a genus one subsurface of $S_g^1$, and that $x_2$ has a different orientation than in Figure~\ref{1bdry}.}
\label{2bdry}
\end{figure}

\textbf{Palindromic and hyperelliptic Torelli groups.} The main focus of our study in this paper is the palindromic Torelli group, $\ptor_n$. This group arises as a natural analogue of a subgroup of $\SMod(S_g^1)$. The \emph{Torelli subgroup} of $\Mod(S_g^1)$, denoted $\I_g^1$, consists of mapping classes that act trivially on $H_1(S_g^1, \Z)$. There is non-trivial intersection between $\I_g^1$ and $\SMod(S_g^1)$; we define $\SI_g^1 := \SMod(S_g^1) \cap \I_g^1$ to be the \emph{hyperelliptic Torelli group}. Brendle--Margalit--Putman \cite{BMP} recently proved a conjecture of Hain \cite{Hain}, also stated by Morifuji \cite{Morifuji}, showing that $\SI_g^1$ is generated by Dehn twists about separating simple closed curves of genus 1 and 2 that are fixed by $s$.

Our generating set for $\ptor_n$ compares favourably with Brendle--Margalit--Putman's for $\SI_g^1$, in the following way. We shall see in Section 2 that any Dehn twist about a symmetric separating curve of genus one that lies in the pre-image of the Nielsen embedding discussed above, maps to a separating $\pi$-twist in $\ptor_{2g}$. In fact, up to conjugation by $\pia_{2g}$, this is the definition of a separating $\pi$-twist. The Dehn twist about the curve $C$ shown in Figure~\ref{2bdry} is an example of such a mapping class. Note that the Dehn twist about $C$ is one of the generators of Brendle--Margalit--Putman. We shall see in Proposition \ref{necgen} that doubled commutator transvections do not suffice to generate $\ptor_n$, so we observe that our generating set involves Brendle--Margalit--Putman's generators in a significant way. Thus, the similarity between $\SI_g^1$ and $\ptor_n$ is not just a superficial comparison of definitions: the Nielsen embedding gives rise to a deeper connection between these two groups.

One way in which the analogy between $\ptor_n$ and $\SI_g^1$ breaks down, however, is their behaviour when $\pia_n$ and $\SMod(S_g^1)$ are abelianised, to $(\Z/2)^3$ and $\Z$ respectively. An immediate corollary of Theorem \ref{mainthm} is that $\ptor_n$ vanishes in the abelianisation of $\pia_n$. In contrast, the image of $\SI_g^1$ in the abelianisation of $\SMod(S_g^1)$ is $4\Z$, which may be shown by calculating the images of Brendle--Margalit--Putman's generators in the abelianisation of $\SMod(S_g^1)$.

\textbf{Palindromes in right-angled Artin groups.} In a forthcoming paper with Anne Thomas \cite{FT14}, we extend Collins' definition of palindromic automorphisms to the right-angled Artin group setting. We obtain generating sets for the analogously defined palindromic automorphism group and palindromic Torelli group of an arbitrary right-angled Artin group. 

\textbf{Approach of the paper.} To prove Theorem \ref{mainthm}, we employ a standard, geometric technique: we find a sufficiently connected complex on which $\ptor_n$ acts with sufficiently connected quotient, and use a theorem of Armstrong \cite{Armstrongorbit} to conclude that $\ptor_n$ is generated by the action's vertex stabilisers. This approach is modelled on a proof of Day--Putman \cite{DayPutman}, which recovers Magnus' finite generating set for the Torelli subgroup of $\aut(F_n)$.

\textbf{Conventions.} We apply functions from right to left. For $g, h \in G$ a group, we let $[g,h] = ghg^{-1}h^{-1}$. In a graph, we denote an edge between vertices $x$ and $y$ by $x-y$. In a group $G$, we will also conflate a relation $P = Q$ with the relator $PQ^{-1}$ when this is unambiguous.

\textbf{Outline of the paper.} In Section 2, the definitions of the palindromic automorphism group and palindromic Torelli group of a free group are given, along with some elementary properties of these groups. In Section 3, we introduce the complex of partial $\pi$-bases of $F_n$, and use it to obtain a generating set for $\ptor_n$. In Section 4, we prove key results about the connectivity of the complexes involved in the proof of Theorem \ref{mainthm}. In Section 5, we obtain a finite presentation for $\Gamma_3[2]$ used in the base case of our inductive proof of Theorem \ref{mainthm}.

\textbf{Acknowledgements.} The author thanks his PhD supervisor Tara Brendle for her guidance, and for introducing him to the palindromic Torelli group. The author is grateful for the hospitality of the Institute for Mathematical Research at Eidgen\"{o}ssische Technische Hochschule Z\"{u}rich, where part of this work was completed. The author also thanks Ruth Charney and Karen Vogtmann for helpful discussions, and Thomas Church, Dan Margalit, Luis Paris, Andrew Putman, Richard Wade and Liam Watson for useful comments on an earlier version of this paper. The author is grateful to the anonymous referee for their numerous constructive comments.

\section{The palindromic automorphism group} Let $F_n$ be the free group of rank $n$, on some fixed free basis $X: = \{x_1, \ldots, x_n \}$. For a word $w = l_1 \ldots l_k$ on $X^{\pm 1}$, let $w^{\rev}$ denote the \emph{reverse} of $w$; that is, we have $ w^{\rev} = l_k \ldots l_1$. Such a word $w$ is said to be a \emph{palindrome} on $X$ if $w^{\rev} = w$. For example, $x_1$, ${x_2}^2$ and $x_2 x_1^{-1} x_2$ are all palindromes on $X$.

An automorphism $\alpha \in \aut(F_n)$ is said to be \emph{palindromic} (with respect to the fixed free basis $X$) if for each $x_i \in X$, the word $\alpha(x_i)$ may be written as a palindrome on $X$. Such automorphisms form a subgroup of $\aut(F_n)$ which we call the \emph{palindromic automorphism group of $F_n$} and denote by $\pia_n$. That $\pia_n$ is a group is easily shown by verifying that $\pia_n$ is the centraliser in $\aut(F_n)$ of the automorphism $\iota$ which inverts each member of $X$. The following proposition gives us information about the form of the palindromes $\alpha(x_i)$.

\begin{prop}\label{oddlengths} Let $\alpha \in \pia_n$ and $x_i \in X$. Then $\alpha(x_i) = w^{\rev} \sigma(x_i)^{\epsilon_i} w$, where $w$ is a word on $X^{\pm1}$, $\sigma$ is a permutation of $X$ and $\epsilon_i \in \{ \pm 1\}$. \end{prop}
\begin{proof}For a palindrome $p = w^{\rev} x_i^{\epsilon_i} w \in F_n$ of odd length ($w \in F_n$, $x_i \in X$, $\epsilon_i \in \{ \pm 1\}$), let $c(p) = x_i$. The following argument is implicit in the work of Collins \cite{Collins}.

Let $\alpha \in \pia_n$. Since $\alpha(X)$ is a free basis, its image under the natural surjection $F_n \to (\Z /2)^n$ must suffice to generate $(\Z/2)^n$. If some $\alpha(x_i)$ is of even length, it will have zero image, and so the image of $\alpha(X)$ could not generate. If $c(\alpha(x_i)) = c(\alpha(x_j))$ for some $i \neq j$, then $\alpha(x_i)$ and $\alpha(x_j)$ will have the same image in $(\Z/2)^n$, and so again $\alpha(X)$ could not generate. \end{proof}

\textbf{Finite generation of $\mathbf{\pia_n}$.} Collins first studied the group $\pia_n$, giving a finite presentation for it. For $i \neq j$, let $P_{ij} \in \pia_n$ map $x_i$ to $x_jx_ix_j$ and fix $x_k$ with $k \neq i$. For each $1 \leq j \leq n$, let $\iota_j \in \pia_n$ map $x_j$ to $x_j^{-1}$ and fix $x_k$ with $k \neq j$. We refer to $P_{ij}$ as an \emph{elementary} palindromic automorphism and to $\iota_j$ as an \emph{inversion}. We let $\Omega^{\pm1}(X)$ denote the group generated by the inversions and the permutations of $X$. The group generated by all elementary palindromic automorphisms and inversions is called the \emph{pure palindromic automorphism group} of $F_n$, and is denoted $\ppia_n$.

Collins showed that $\pia_n \cong \epia_n \rtimes \Omega^{\pm1}(X)$ for $n \geq 2$, where $\epia_n = \langle P_{ij} \rangle$. The group $\Omega^{\pm1}(X)$ acts on $\epia_n$ in the natural way, and a defining set of relations for $\epia_n$ is given by \begin{enumerate}\item $[P_{ik}, P_{jk}] = 1$,
\item $[P_{ij}, P_{kl}] = 1$, and
\item $P_{ij} P_{jk} P_{ik} = P^{-1}_{ik}P_{jk}P_{ij}$,
\end{enumerate}
where $i,j,k,l$ are pairwise different and the obviously undefined relators are omitted in the $n=2$ and $n=3$ cases.

We remark that, as noted by Collins \cite{Collins}, this presentation of $\epia_n$ is very similar to one given for the \emph{pure symmetric automorphism group of $F_n$}, $\mathrm{P}\Sigma A_n$, which consists of automorphisms taking each $x \in X$ to a conjugate of itself. This similarity is not entirely surprising, as we may think of a palindrome $yxy$ as a conjugate $yxy^{-1}$, working `mod 2' ($x,y \in X$). The embedding $B_n \hookrightarrow \pia_n$ discussed in Section 1 bears a striking resemblance to Artin's faithful representation of $B_n$ into $\Sigma A_n$, the full symmetric automorphism group, whose members take each $x \in X$ to \emph{some} conjugate \cite[Corollary 1.8.3]{birman}; this similarity arises via the branched double cover map $S_g^1 \to D_{2g+1}$ \cite[Figure 9.13]{primer}.

\begin{figure}
\centering
\def\svgwidth{5.5in}
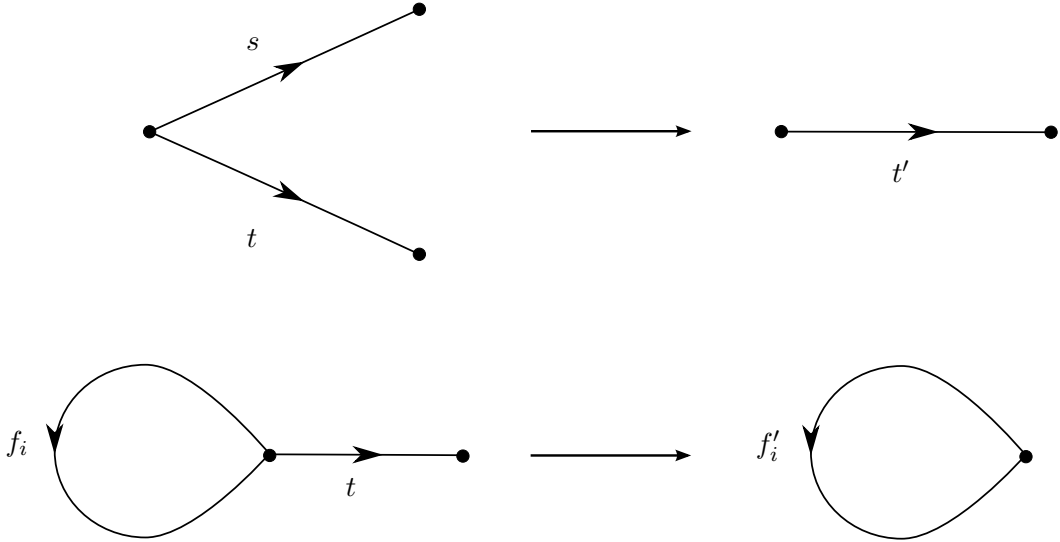
\caption{The two types of folding that may occur for our graph morphism $\phi$. Wade \cite{Wade} refers to the top fold as a type 1 fold, and to the bottom as a type 2 fold. The edges are labelled suggestively: we will demand that $s, t \in T$ and $f_i \not \in T$.}
\label{folds}
\end{figure}

Using graph folding techniques of Stallings, we obtain a new proof of finite generation of $\pia_n$, as well as finding generating sets for certain fixed point subgroups of $\pia_n$. We first introduce the notation and terminology of Wade \cite{Wade} regarding graph folding.

Let $R_n$ denote the wedge of $n$ copies of $S^1$ at a point $o$. We canonically identify $\pi_1(R_n, o)$ with $F_n$ by selecting an orientation of each $S^1$, and labelling the $i$th copy of $S^1$ by $x_i \in X$. Note, we shall let $\bar x_i$ denote the edge obtained by reversing the orientation of $x_i$.

Now, let $Y$ be a finite graph of rank $n$ with basepoint $b$. We will view our graphs as combinatorial objects, rather than topological ones. In particular, morphisms between graphs must take edges to edges, rather than edge-paths. A free basis for the (free) fundamental group $\pi_1(Y,b)$ is obtained in the usual way, by selecting a maximal tree $T$ in $Y$, then choosing an orientation of the edges $f_1, \ldots , f_n$ in $Y$ but not $T$. To be consistent with Wade, we canonically orient an edge $e$ of $T$ by declaring its initial vertex $i(e)$ to be the one closer to the basepoint $b$ under the edge-path metric on $T$.

Suppose $\theta : Y \to R_n$ is a morphism of graphs that induces an isomorphism of fundamental groups. The morphism $\theta$, together with the choice of basepoint $b$, maximal tree $T$ and an ordering $L$ of the (oriented) edges of $Y \setminus T$ form a \emph{branding} of the graph $Y$. A graph $Y$ together with a 4-tuple $\mathcal{G} = (b,T,L, \theta)$ form a \emph{branded graph} with \emph{branding} $\mathcal{G}$.

Each branded graph $Y$ with branding $\mathcal{G} = (b,T,L, \theta)$ yields an automorphism $B_\mathcal{G}~\in~\autfn$, as follows. For each $x_i$ in the free basis $X$ of $F_n$, we have \[ B_{\mathcal{G}} (x_i) = \theta_\ast(y_i) ,\] where $\{ y_1, \dots, y_n \}$ is the free basis of $\pi_1(Y,b)$ arising from the choices of $b$, $T$ and $L$ in the branding $\mathcal{G}$, and $\theta_\ast : \pi_1(Y,b) \to \pi_1(R_n, o)$ is the map induced by $\theta$.

If the morphism $\theta$ maps a pair of edges $e_1$ and $e_2$ with $i(e_1) = i(e_2)$ to the same edge $l$ of $R_n$, then $\theta$ factors through the quotient graph $Y'$ of $Y$ obtained by \emph{folding} $e_1$ and $e_2$ together: that is, the graph obtained by identifying $e_1$ with $e_2$, and also their terminal vertices, $t(e_1)$ and $t(e_2)$, with each other. In particular, if $q : Y \to Y'$ is the quotient map obtained by the folding, then there is a unique graph morphism $\theta ' : Y' \to R_n$ such that $\theta = \theta ' \circ q$. While Stallings considered more general foldings, since we require $\theta$ to induce an isomorphism of fundamental groups, only two types of folding may arise for us, which are shown in Figure~\ref{folds}.

If we insist that the edges $s$ and $t$ seen in Figure \ref{folds} lie in $T$, and that the edge $f_i$ does not, carrying out either type of fold induces a branding $\mathcal{G}'$ of the folded graph $Y'$ (it is non-trivial to verify that the image of $T$ in $Y'$ is a maximal tree; we leave this to Wade). It may also be the case that we wish to carry out a fold of type 1 or type 2, but that $s$ or $t$ does not lie in $T$. Before folding, we must change maximal tree so that the relevant edges lie in the new tree. This defines a new branding $\mathcal{G}''$ of $Y$. In either case, it may be shown via a careful consideration of $\pi_1(Y,b)$ (see \cite[Propositions 3.2, 3.3]{Wade}) that $B_\mathcal{G} = B_{\mathcal{G}'} \cdot W'$ and $B_\mathcal{G} = B_{\mathcal{G}''} \cdot W''$, where $W'$ and $W''$ are specified \emph{Whitehead automorphisms} of $F_n$. These are automorphisms which fix some $x \in X$ and send each $x_i \in X \setminus \{x \}$ to one of $x_i$, $x_ix^{\epsilon_i}$, $x^{\epsilon_i}x_i$ or $x^{\epsilon_i} x_i x^{- \epsilon_i}$ for some $\epsilon_i \in \{ \pm 1 \}$.

Stallings' folding algorithm allows us to repeatedly fold the graph $Y$ and its quotients, beginning with the morphism $\theta : Y \to R_n$, then continuing to fold via $\theta ' : Y' \to R_n$, and so on. This procedure eventually terminates when we exhaust the edges we are able to fold; in this case, Stallings showed that the quotient graph is $R_n$, and so the morphism $\psi : R_n \to R_n$ obtained by repeatedly folding via $\theta$ simply permutes and perhaps inverts the $n$ loops in $R_n$. This folding procedure allows us to write the automorphism $B_{\mathcal{G}}$ we began with as a product of Whitehead automorphisms, and permutations and inversions of $X$.

With the details of folding established, we now put the algorithm to use to find generators for $\pia_n$.

\begin{prop}\label{palgens}Fix $ 0 \leq k \leq n$, and let $\pia_n(k)$ consist of automorphisms which fix $x_1, \ldots, x_k$. (Our convention is that $\pia_n(0) = \pia_n$). A finite generating set for $\pia_n(k)$ is \[ \left [ \Omega^{\pm1}(X) \cap \pia_n(k) \right ] \cup \{P_{ij} \mid i > k \}. \] \end{prop}
\begin{proof} The idea behind this proof was inspired by a proof of Wade \cite[Theorem 4.1]{Wade}. 

We begin by introducing some terminology. Let $\phi : S \to T$ be an isomorphism of finite trees. For a vertex (resp. edge) $r$ of $S$, denote by $r'$ the image of $r$ under $\phi$. Choose a distinguished vertex $v$ of $S$, of valence 1. An \emph{arch of $S$ at $v$} (see Figure \ref{arch}) is the graph formed by gluing $S$ to $T$ along $v$ and $v'$, then for each vertex $r \in S$, adding some (possibly zero) number of edges between $r$ and $r'$ (note, we allow $r=v$). We refer to these new edges as \emph{bridges}. The image of $v$ in the arch forms a natural base point, and any edge with $v$ as one of its endpoints is called a \emph{stem}. By a \emph{wedge of arches} we mean a collection of arches glued together at their base points. Note that each of the trees $S_i$ and $T_i$ of each arch sit inside $Y$ as subgraphs, and $Y$ is the union of these subgraphs, together with any bridges inside each arch.

Let $\theta : Y \to R_n$ be a graph morphism, with $Y$ a wedge of arches. We call $\theta$ \emph{symmetric} if for each edge $s_i$ in each tree $S_i$ in each arch of $Y$, we have $\theta(s_i ') = \theta (\bar {s_i})$. We shall define two new types of folding that we may carry out to any symmetric morphism $\theta : Y \to R_n$, with the resulting morphism $\theta ' : Y' \to R_n$ on the folded graph $Y'$ also being symmetric.

\begin{figure}
\centering
\def\svgwidth{2in}
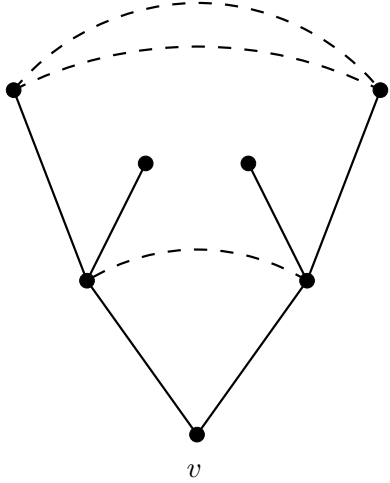
\caption{An example of an arch, with base point $v$. The dashed edges indicate the bridges that have been added to the trees that were glued together at the base point.}
\label{arch}
\end{figure}

Let $\alpha \in \pia_n(k)$. We may realise $\alpha$ as a morphism of graphs $\theta : Z \to R_n$, where $Z$ is the result of subdividing each $S^1$ of $R_n$ into the appropriate number of edges, and `spelling out' the word $\alpha(x_i)$ on the $i$th copy of $S^1$. Precisely, the $j$th edge of the oriented, subdivided $S^1$ corresponding to $\alpha(x_i)$ is mapped to the loop in $R_n$ corresponding to the $j$th letter of $\alpha(x_i)$, correctly oriented. Note that $Z$ is a wedge of arches, and $\theta$ is symmetric by construction. We thus have $\alpha = B_{\mathcal{G}}$ where $\mathcal{G}$ is the branding of $Z$ arising from the maximal tree that excludes the (appropriately ordered) middle subdivided edge of each copy of $S^1$. We now use graph folding to write $\alpha$ as a product of permutations, inversions and elementary palindromic automorphisms.

Let $\theta : Y \to R_n$ be symmetric, for some wedge of arches $Y$, built out of trees $S_i$, $T_i$ ($1 \leq i \leq k$). Since $\theta$ is symmetric, foldings of $Y$ come together in natural pairs. Consider folds of type 1. For instance, if we are able to fold together two edges $h_i \in S_i$ and $h_j \in S_j$ since $\theta(h_i) = \theta(h_j)$ (allowing $i = j$), then we will also be able to fold together $h_i'$ and $h_j'$, as they will also both have the same image under $\theta$, namely $\theta(\bar {h_i}) = \theta (\bar {h_j})$. We call this pair of folds a \emph{type A 2-fold}.

We may also have a sequence of edges $(h_{j-1},h_j,h_{j+1})$ mapped under $\theta$ to the sequence $(\bar x,x,\bar x)$ where $x$ is an oriented edge of $R_n$, $h_{j-1} \in S_i$, $h_{j+1} = h_{j-1} '$ and $h_j$ is a bridge. We fold $h_{j-1}$ and $h_{j+1}$ onto $h_j$, and call this pair of folds a \emph{type B 2-fold}. Such a fold is seen in Figure \ref{bfold}.
  
\begin{figure}
\centering
\def\svgwidth{5.5in}
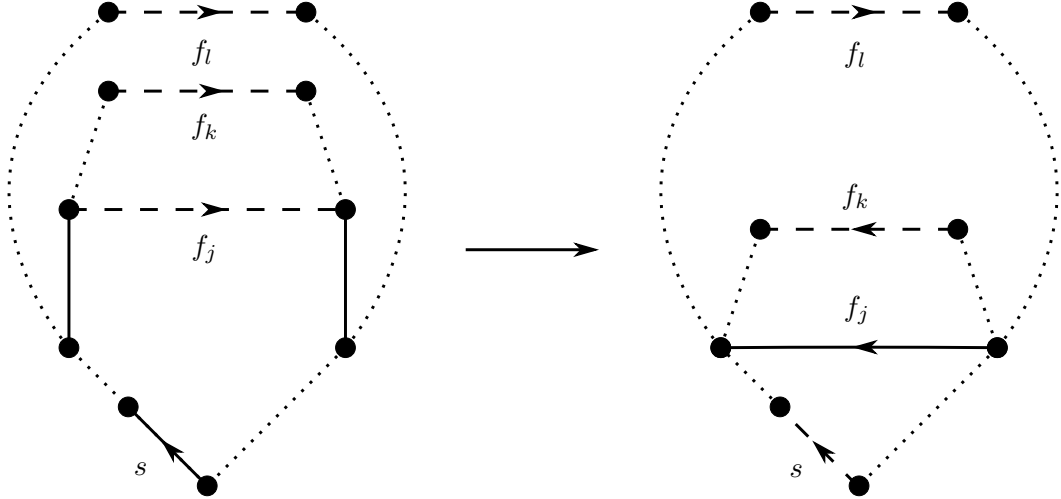
\caption{The two adjacent solid edges are folded onto $f_j$. The dashed edges represent edges excluded from the graph's chosen maximal tree. In order to record what effect this type B 2-fold has on the branded graph's associated automorphism, we must swap $f_j$ into the maximal tree, in place of the stem $s$. }
\label{bfold}
\end{figure}
Doing either of these 2-folds to $Y$ yields another, different wedge of arches, $Y'$, say. A type B 2-fold simply removes an edge of valence one from $S_i$ (and its corresponding edge in $T_i$) by folding it onto a bridge, producing new trees $S_i'$ and $T_i'$ which we use to construct $Y'$ as a wedge of arches. A type A 2-fold similarly alters the trees $S_i$, $S_j$, $T_i$ and $T_j$, producing new trees $S_i'$ and $T_i'$ in a description of $Y'$ as a wedge of arches. The morphism $\theta ' : Y' \to R_n$ induced by the folding of $Y$ is again symmetric: any edges $s_i$ and $s_i'$ that were not folded still satisfy $\theta ' (s_i') = \theta' (\bar{s_i})$ by construction of $\theta '$, but so do the images of any folded edges, given how we decompose $Y'$ as a wedge of arches using the new trees $S_i'$ and $T_i'$.

In order to see what effect these 2-folds have on $\alpha \in \pia_n$, we must keep track of a preferred maximal tree $T$ we define on each wedge of arches $Y$. The edges of $Y$ not in $T$ are the bridges coming from each arch. In order to carry out a type B 2-fold we must swap the bridge $f_j$ (seen in Figure~\ref{bfold}) into the maximal tree. Let $p_{i(f_j)}$ denote the unique reduced path in $T$ joining the base point to the initial vertex of $f_j$. Apart from one degenerate case, which we deal with separately, we may always swap $f_j$ into the maximal tree $T$ by excluding the stem appearing in $p_{i(f_j)}$. Using calculations of Wade \cite[Propositions 3.2, 3.3]{Wade}, it is straightforward to verify that the effect of swapping maximal trees in this way, doing a type B 2-fold, then swapping back to the maximal tree where all bridges are excluded is to carry out an elementary palindromic automorphism $P_{ij}^{\epsilon_k}$ to some members of $X$. Precisely, let $\theta: Y_1 \to R_n$ be a symmetric morphism of graphs, where $Y_1$ has branding $\mathcal{G}_1$ and let $\mathcal{G}_2$ be the induced branding of the graph $Y_2$ obtained by carrying out the above series of tree swaps and folds. Then
\[\phi_{\mathcal{G}_1} = \phi_{\mathcal{G}_2} \cdot P, \] where $\phi_{\mathcal{G}_i}$ is the automorphism of $F_n$ associated to $\mathcal{G}_i$ ($i = 1,2$) and $P$ is a product of elementary palindromic automorphisms.

The only degenerate case of the above is when one (and hence both) of the edges we want to fold onto a bridge is a stem. In this case, we do one of two things. If the bridge is a loop at the base point $v$, we carry out two type 2 folds. Otherwise, we change maximal trees as before then fold one of the stems onto the bridge with a type 1 fold. This causes the other stem to become a loop, around which we fold the bridge using a type 2 fold. As before, the automorphism of $F_n$ associated to these sequences of steps is a product of elementary palindromic automorphisms.

Carrying out a sequence of 2-folds of types A and B eventually produces a map $R_n \to R_n$, and so we complete the folding algorithm by applying the appropriate automorphism from $\Omega^{\pm1}(X)$. Notice that since $\alpha \in \pia_n(k)$, the graph $Z$ we constructed has a single loop at the base point for each $x_i$ ($1 \leq i \leq k$), as $\alpha(x_i) = x_i$, so the first $k$ ordered loops of $R_n$ were not subdivided to form $Z$. Thus, while folding such a graph $Y$, we only need Collins' generators that fix the first $k$ members of the free basis $X$. The proposition is thus proved.
\end{proof}
\begin{corollary}\label{purestabs}The group $\ppia_n(k)$ of pure palindromic automorphisms fixing $x_1, \ldots, x_k$ $(0 \leq k \leq n)$ is generated by the set $\{P_{ij}$, $\iota_i \mid i > k \}$. \end{corollary}

\textbf{The principal level 2 congruence subgroup of $\gln$.} Recall that $\Gamma_n[2]$ denotes the \emph{principal level 2 congruence subgroup} of $\gln$, that is, kernel of the map $\gln \to \mathrm{GL}(n, \Z_2)$ given by reducing matrix entries mod 2. Let $S_{ij}$ be the matrix with 1s on the diagonal, 2 in the $(i,j)$ position and 0s elsewhere, and let $O_i$ be the matrix which differs from the identity matrix only in having a $-1$ in the $(i,i)$ position. The following lemma verifies a well-known generating set for $\Gamma_n[2]$ (see, for example, McCarthy--Pinkall \cite[Corollary 2.3]{McCarthyPinkall}). We include a proof here to introduce the idea of an `even division algorithm', which we utilise in the proof of Theorem \ref{augment}.
\begin{lemma}\label{conggens}The set $\{O_i, S_{ij} \mid 1 \leq i \neq j \leq n \}$ generates $\Gamma_n[2]$. \end{lemma}
\begin{proof}Observe that we may think of the matrices $S_{ij}$ as corresponding to carrying out `even' row operations, that is, adding an even multiple of one matrix row to another. Let $u$ be the first column of some matrix in $\Gamma_n[2]$, and denote by $u^{(i)}$ the $i$th row of $u$. Let $v_1$ be the standard column vector with a 1 in the first entry and 0s elsewhere.

\textbf{Claim:} The column $u$ can be reduced to $\pm v_1$ using even row operations.

We use induction on $|u^{(1)}|$. For $|u^{(1)}| = 1$, the claim is obvious. Now suppose $|u^{(1)}| > 1$. As in the proof of Proposition \ref{oddlengths}, we deduce that there must be some $u^{(j)}$ which is not a multiple of $u^{(1)}$. By the Division Algorithm, there exist $q, r \in \Z$ such that $u^{(j)} = q|u^{(1)}| + r$, with $0 \leq r < |u^{(1)}|$. If $q$ is not even, we instead write $u^{(j)} = (q+1)|u^{(1)}| + (r-|u^{(1)}|)$. Note that if $q$ is odd, then $r \neq 0$, since $u^{(1)}$ is odd and $u^{(j)}$ is even, and so $-|u^{(1)}| < r- |u^{(1)}|$. Depending on the parity of $q$, we do the appropriate number of even row operations to replace $u^{(j)}$ with $r$ or $r - |u^{(1)}|$. In both cases, we have replaced $u^{(j)}$ with an integer of absolute value smaller than $|u^{(1)}|$. It is clear that now we may reduce the absolute value of $u^{(1)}$ by either adding or subtracting twice the (new) $j$th row from the first row, and so by induction we have proved the claim.

We now induct on $n$ to prove the lemma. It is clear that $\Gamma_1[2] = \langle O_1 \rangle$. Using the above claim, we may assume that we have reduced $M \in \Gamma_n[2]$ so it is of the form
\[\left [ \begin{array}{c|c}
 \pm 1 & \ast \\
\hline
0 & N  \\
\end{array} \right ], \] where $N \in \Gamma_{n-1}[2]$. Our aim is to further reduce $M$ to the identity matrix using the set of matrices in the statement of the lemma. By induction, we may assume that $N$ can be reduced to the identity matrix using the appropriate members of $\{S_{ij}, O_i \mid i, j > 1 \}$. Then we simply use even row operations to fix the top row, and finish by applying $O_1$ if necessary.
\end{proof}

By Lemma \ref{conggens}, the restriction of the canonical map $\aut(F_n) \to \gln$ gives the short exact sequence \[ 1 \longrightarrow \ptor_n \longrightarrow \ppia_n \longrightarrow \Gamma_n[2] \longrightarrow 1, \] since $P_{ij}$ maps to $S_{ji}$ and $\iota_i$ maps to $O_i$.

The rest of the paper is concerned with finding a generating set for the palindromic Torelli group $\ptor_n$. In order to describe our generating set, we introduce some terminology. 

Let $Y$ be the image of the free basis $X$ under some automorphism $\alpha \in \pia_n$. The set $Y$ is also a free basis for $F_n$, whose members are palindromes on $X$; thus, we refer to $Y$ as a \emph{$\pi$-basis}. An automorphism $\phi \in \ptor_n$ is a \emph{doubled commutator transvection} if for some $y_1$, $y_2$, $y_3$ in some $\pi$-basis $Y$, $\phi$ maps $y_1$ to $[y_2,y_3]^{\mathrm{rev}} y_1 [y_2, y_3]$, and fixes the other members of $Y$. Observe that $\phi \in \ptor_n$ is a doubled commutator transvection if and only if $\phi$ is conjugate in $\pia_n$ to the commutator $\chi_1 := [P_{12}, P_{13}]$. An automorphism $\phi \in \ptor_n$ is a \emph{separating $\pi$-twist} if for some $y_1$, $y_2$, $y_3$ in some $\pi$-basis $Y$, $\phi$ is given by
\[ \phi(y_i) = \begin{cases} d^{\mathrm{rev}} y_1 d, & \mbox{ if } i=1, \\
d^{-1} y_2 {(d^{\mathrm{rev}})}^{-1}, & \mbox{ if } i=2, \\
d^{\mathrm{rev}} y_3 d, &  \mbox{ if } i = 3, \\
y_i, & \mbox{ otherwise}, \end{cases} \]
where $d = {y_1}^{-1}{y_2}^{-1}{y_3}^{-1}y_1y_2y_3 \in F_n$. It is a straightforward, if lengthy, calculation to verify that $\phi \in \ptor_n$ is a separating $\pi$-twist if and only if $\phi$ is conjugate in $\pia_n$ to the automorphism \[\chi_2 := (P_{23}{P_{13}}^{-1}P_{31}P_{32}P_{12}{P_{21}}^{-1})^2 \in \ptor_n.\]

The definition of a separating $\pi$-twist may seem unwieldy, however it belies a hidden geometry. The automorphism $\chi_2$ is the image in $\ptor_n$ under the Nielsen embedding of the Dehn twist about the curve $C$, seen in Figure~\ref{2bdry}. We call such automorphisms separating $\pi$-twists to reflect this geometric interpretation.

Theorem~\ref{mainthm} states that doubled commutator transvections and separating $\pi$-twists suffice to generate $\ptor_n$. To prove this, we construct a new complex on which $\ptor_n$ acts in a suitable way. We then apply a theorem of Armstrong \cite{Armstrongorbit} to conclude that $\ptor_n$ is generated by the action's vertex stabilisers. In the following section, we define the complex and use it to prove Theorem \ref{mainthm}.

\section{The complex of partial $\pi$-bases} Day--Putman \cite{DayPutman} use the \emph{complex of partial bases} of $F_n$, denoted $\mathcal{B}_n$, to derive a generating set for $\IA_n$. We build a complex modelled after $\mathcal{B}_n$, and follow the approach of Day--Putman to find a generating set for $\ptor_n$.

Fix $X := \{x_1, \ldots, x_n\}$ as a free basis of $F_n$. A \emph{$\pi$-basis}, as discussed above, is a set of palindromes on $X$ which also forms a free basis of $F_n$. A \emph{partial $\pi$-basis} is a set of palindromes on $X$ which may be extended to a $\pi$-basis. The \emph{complex of partial $\pi$-bases} of $F_n$, denoted $\pbc_n$, is defined to be the simplicial complex whose $(k-1)$-simplices correspond to partial $\pi$-bases $\{w_1, \ldots, w_k \}$. We postpone until Section 4 the proof of the following theorem on the connectedness of $\pbc_n$.

\begin{theorem}\label{pbcconn}For $n \geq 3$, the complex $\pbc_n$ is simply-connected.
\end{theorem}

Our complex $\pbc_n$ is not a subcomplex of $\mathcal{B}_n$, as the vertices of $\mathcal{B}_n$ are taken to be conjugacy classes, rather than genuine members of $F_n$. We remove this technicality, as it can be shown that two odd length palindromes are conjugate if and only if they are equal. Given this, it is clear, however, that $\pbc_n$ is isomorphic to a subcomplex of $\mathcal{B}_n$.

There is an obvious simplicial action of $\pia_n$ on $\pbc_n$. This action is, by definition, transitive on the set of $k$-simplices, for each $0 \leq k < n$. Further, $\ptor_n$ acts without rotations, that is, if $\phi \in \ptor_n$ stabilises a simplex $s$ of $\pbc_n$, then it fixes $s$ pointwise. Following work of Charney~\cite{Charney} on related complexes, we obtain that the quotient of $\pbc_n$ by $\ptor_n$ is highly-connected.
\begin{theorem}\label{quoconn}For $n \geq 3$, the quotient $\pbc_n / \ptor_n$ is $(n-3)$-connected.\end{theorem}
The proof of this theorem is discussed in Section 4.

Theorems \ref{pbcconn} and \ref{quoconn} allow us to apply the following theorem of Armstrong \cite{Armstrongorbit} to the action of $\ptor_n$ on $\pbc_n$, for $n \geq 4$. The statement of the theorem is as given in Day--Putman \cite{DayPutman}.
\begin{theorem}\label{armstrong} Let $G$ act simplicially on a simply-connected simplicial complex $X$, without rotations. Then $G$ is generated by the vertex stabilisers of the action if and only if $X / G$ is simply-connected. \end{theorem}

We analyse the vertex stabilisers of $\ptor_n$ using an inductive argument. It is known that $\ptor_1 = 1$ and $\ptor_2 = 1$; the latter equality follows from the fact that $\IA_2 = \inn(F_2)$ and $\inn(F_n) \cap \pia_n = 1$ for $n \geq 1$. We treat the $n=3$ case differently, as the quotient $\pbc_3 / \ptor_3$ is not simply-connected, and so does not allow us to apply Armstrong's theorem directly. This treatment is postponed until Section 5.

\textbf{A Birman exact sequence.} We require a version of the free group analogue of the Birman exact sequence, as developed by Day--Putman \cite{DayPutmanBirseq}. Recall that $\ppia_n(k)$ consists of the pure palindromic automorphisms fixing $x_1, \ldots, x_k$. \begin{prop}\label{birman}For $0 \leq k \leq n$, there exists the split short exact sequence \[ 1 \longrightarrow \mathcal{J}_n(k) \longrightarrow \ppia_n(k) \longrightarrow \ppia_{n-k} \longrightarrow 1, \] where $\mathcal{J}_n(k)$ is the normal closure in $\ppia_n(k)$ of the set $\{P_{ij} \mid i > k,$ $j \leq k \}$. \end{prop}
\begin{proof}A map $\theta_\ast: \ppia_n(k) \to \ppia_{n-k}$ is induced by the map $\theta: F_n \to F_{n-k}$ that trivialises each $x_1, \dots, x_k$. Let $\{ y_{k+1} , \dots , y_n \}$ be a free basis for $F_{n-k}$, where $\theta(x_i) = y_i$ for $k+1 \leq i \leq n$. Denote by $Q_{ij}$ and $\eta_i$ the elementary palindromic automorphism sending $y_i$ to $y_j y_i y_j$, and the inversion sending $y_i$ to ${y_i}^{-1}$, respectively ($k+1 \leq i \neq j \leq n$).

By Corollary~\ref{purestabs}, we know that $\ppia_n(k)$ is generated by the set \[\mathcal{S} := \{P_{ij}, \iota_i \mid i > k, 1 \leq j \leq n \}.\] If $j \leq k$, then $\theta_\ast(P_{ij})$ is trivial. If $i,j \geq k+1$, then $\theta_\ast(P_{ij}) = Q_{ij}$ and $\theta_\ast(\iota_i) = \eta_i$, so we have that $\theta_\ast$ is surjective, by examining Collins' generators for $\ppia_{n-k}$. Indeed, the map $\theta_\ast$ has a section, taking $Q_{ij}$ to $P_{ij}$ and $\eta_i$ to $\iota_i$, which we know is well-defined by Collins' finite presentation for $\ppia_{n-k}$. Thus, we obtain a split short exact sequence via the epimorphism $\theta_\ast$.

All that is left to establish is the kernel of $\theta_\ast$. Notice that we have a presentation for $\ppia_{n-k}$ in terms of the generating set $\theta_\ast (\mathcal{S})$: explicitly, we add the relations $\theta_\ast(P_{ij}) = 1$ for $j \leq k$ to Collins' relations on the set $\{Q_{ij}, \eta_i \}$. It is a standard fact (see, for example, the proof of Theorem 2.1 in Magnus--Karrass--Solitar \cite{MKS}) that the kernel of $\theta_\ast$ is the normal closure in $\ppia_n(k)$ of the obvious lifts of the defining relators on $\theta_\ast (\mathcal{S})$. The only defining relators with non-trivial lifts in $\ppia_n(k)$ are the relators $\theta_\ast(P_{ij})$ with $j \leq k$, thus the kernel is $\mathcal{J}_n(k)$ as in the statement of the proposition.
\end{proof}

Our `Birman kernel' $\mathcal{J}_n(k)$ is rather worse behaved than the analogous Birman kernel of Day--Putman. Their kernel, denoted $\mathcal{K}_{n,k,l}$, is finitely generated, whereas it may be shown by adapting the proof of their Theorem E that $\mathcal{J}_n(k)$ is not. This difference is due in part to the fact that their version of $\ppia_n(k)$ need only fix each of $x_1,\ldots,x_k$ up to conjugacy. The lack of finite generation of $\mathcal{J}_n(k)$ is, however, not an obstacle to the goal of the current paper; we only require that $\mathcal{J}_n(k)$ is \emph{normally} generated by a finite set.

Our Birman exact sequence projects into $\gln$ in an obvious way, made precise in the following lemma. Let $v_i$ denote the image of $x_i \in F_n$ under the abelianisation map. We denote by $\Gamma_n[2](k)$ the members of $\Gamma_n[2]$ which fix $v_1, \ldots, v_k \in \Z^n$, and by $\mathcal{H}_n(k)$ the group $\mathrm{Hom}(\Z^{n-k}, (2\Z)^k)$.
\begin{lemma}\label{commdiag}Fix $0 \leq k \leq n$. Then there exists the following commutative diagram of split short exact sequences\[ \xymatrix{
1 \ar[r] & \mathcal{J}_n(k) \ar[r] \ar@{->>}[d] & \ppia_n(k) \ar[r] \ar@{->>}[d] & \ar@/^1pc/[l]^s \ppia_{n-k}  \ar[r] \ar@{->>}[d] & 1\\
1 \ar[r] & \mathcal{H}_n(k)   \ar[r]      & \Gamma_n[2](k)  \ar[r]         & \ar@/^1pc/[l]^t \Gamma_{n-k}[2] \ar[r]      & 1}, \] where $s$ and $t$ are the obvious splitting homomorphisms.
\end{lemma}
\begin{proof}The top row is given by Proposition \ref{birman}. A generating set for $\Gamma_n[2](k)$ follows from the proof of Lemma \ref{conggens}; it is precisely the image in $\gln$ of $\{P_{ij}$, $\iota_i \mid i > k \}$, the generating set of $\ppia_n(k)$ given by Corollary \ref{purestabs}. The bottom row then follows by an argument similar to the proof of Proposition \ref{birman}, noting that the kernel is generated by the images of $P_{ij}$ $(i > k$, $j \leq k)$. It is straightforward to verify that this kernel is $\mathrm{Hom}(\Z^{n-k}, (2\Z)^k)$. Intuitively, $\alpha \in \mathrm{Hom}(\Z^{n-k}, (2\Z)^k)$ is encoding how many (even) multiples of $v_j$ $(1 \leq i \leq k)$ are added to each $v_i$ $(k < j \leq n)$.

The only vertical map left to consider is the right-most one. Its existence and surjectivity follow from Lemma \ref{conggens}. It is clear that all the arrows commute, and that the splitting homomorphisms $s$ and $t$ are compatible with the commutative diagram, so the proof is complete.
\end{proof}

\textbf{A generating set for $\mathcal{J}_n(1) \cap \ptor_n$.} By mapping $\ppia_n(k)$ into $\Gamma_n[2](k)$ then conjugating the normal subgroup $\mathcal{H}_n(k)$, we obtain a homomorphism $\alpha_k : \ppia_n(k) \to \aut(\mathcal{H}_n(k))$. Setting $k=1$, we obtain the following lemma.
\begin{lemma}\label{jcap} The group $\mathcal{J}_n(1) \cap \ptor_n$ is normally generated in $\mathcal{J}_n(1)$ by the set
\[\{[P_{ij},P_{i1}], \phantom{i} [P_{ij},P_{j1}]P_{i1}^2 \phantom{\phi} \mid \phantom{\phi} 1 < i \neq j \leq n \}.\] \end{lemma}

\begin{proof}By Lemma \ref{commdiag}, there is a short exact sequence \[ 1 \longrightarrow \mathcal{J}_n(1) \cap \ptor_n \longrightarrow \mathcal{J}_n(1) \longrightarrow \mathcal{H}_n(1) \longrightarrow 1. \] The set $Y := \{\phi P_{j1} \phi^{-1} \mid \phi \in \ppia_n(1), 1 < j \leq n\}$ generates $\mathcal{J}_n(1)$ by Proposition \ref{birman}. Let $a_j$ denote the image of $P_{j1}$ in $\gln$. A direct calculation verifies that the set $\{a_j\}$ is a free abelian basis for $\mathcal{H}_n(1)$.

For $\phi \in \ppia_n(1)$, let $\bar \phi$ denote the image of $\phi$ in $\Gamma_n[2](1)$, and let $\bar Y$ denote the image of $Y$. The set of relations
\[\{[a_i, a_j] = 1, \phantom{i} \bar \phi a_i \bar \phi^{-1} = \alpha_1(\phi)(a_i) \mid 1 < i \neq j \leq n, \phantom{i} \phi \in \ppia_n(1) \},\] together with the generating set $\bar Y$, forms a presentation for $\mathcal{H}_n(k)$. It is clear that the image of any member of $Y$ in $\mathcal{H}_n(1)$ is a word on the free abelian basis $\{a_i \}$, and that this word is determined by the homomorphism $\alpha_1$.

The group $\mathcal{J}_n(1) \cap \ptor_n$ is normally generated in $\mathcal{J}_n(1)$ by the obvious lifts of the (infinitely many) relators in the given presentation for $\mathcal{H}_n(1)$. The relators of the form $[a_i,a_j]$ have trivial lift, and so are not required in the generating set. Let $C$ be the finite generating set for $\ppia_n(1)$ given by Corollary \ref{purestabs}. It can be shown that the obvious lift of the finite set of relators $$ D:= \{\bar c a_j {\bar c}^{-1} \alpha_1(c)(a_j)^{-1} \mid c \in C^{\pm1}, 1 < j \leq n \}$$ suffices to normally generate $\mathcal{J}_n(1) \cap \ptor_n$. This may be seen using a simple induction argument on the length of a given expression of $\phi \in \ppia_n(1)$ on $C^{\pm1}$.

All that remains is to verify that the obvious lift of $D$ is the set given in the statement of the lemma; this is a straightforward calculation.
\end{proof}

\textbf{Proof of Theorem A.} We now prove Theorem A, using the action of $\ptor_n$ on $\pbc_n$.

\begin{proof}[Proof of Theorem A] Recall that the set of doubled commutator transvections in $\ptor_n$ is precisely the conjugacy class of $[P_{12},P_{13}]$ in $\pia_n$, and that the set of separating $\pi$-twists in $\ptor_n$ is precisely the conjugacy class of \[ (P_{23}{P_{13}}^{-1}P_{31}P_{32}P_{12}{P_{21}}^{-1})^2 \] in $\pia_n$.

The group $\ptor_n$ acts on $\pbc_n$ simplicially and without rotations. Combining Theorems \ref{pbcconn}, \ref{quoconn} and \ref{armstrong}, we conclude that for $n \geq 4$, $\ptor_n$ is generated by the vertex stabilisers of the action on $\pbc_n$. 

Let $\ptor_n(1)$ denote the stabiliser of the vertex $x_1$. Since $\pia_n$ acts transitively on the vertices of $\pbc_n$, the stabiliser in $\ptor_n$ of any vertex is conjugate in $\pia_n$ to $\ptor_n(1)$. Lemma \ref{commdiag} gives us the split short exact sequence \[ 1 \longrightarrow \mathcal{J}_n(1) \cap \ptor_n \longrightarrow \ptor_n(1) \longrightarrow \ptor_{n-1} \longrightarrow 1 . \]

We induct on $n$. By the above split short exact sequence, to generate $\ptor_n(1)$ it suffices to combine a generating set of $\mathcal{J}_n(1) \cap \ptor_n(1)$ with a lift of one of $\ptor_{n-1}$.

We begin with the base case, $n=3$. In Section 5, we verify that the presentation of $\Gamma_3[2]$ given in Corollary \ref{congpres} is correct when $n=3$. Given the short exact sequence
\[ 1 \longrightarrow \ptor_3 \longrightarrow \ppia_3 \longrightarrow \Gamma_3[2] \longrightarrow 1, \] we may take the obvious lifts of the relators in this presentation as a normal generating set for $\ptor_3$ in $\ppia_3$. Relators 1--7 are trivial when lifted. Relator 8 lifts to $[P_{ij},P_{ik}]$ and relator 9 lifts to $[P_{jk}, P_{ij}]{P_{ik}}^{-2}$, which equals $P_{ik} [ P_{ij},P_{ik}] {P_{ik}}^{-1}$. Thus the lifts of relators 8 and 9 are conjugate to $[P_{12}, P_{13}]$ in $\pia_3$. Finally, relator 10 lifts to \[ (P_{23}{P_{13}}^{-1}P_{31}P_{32}P_{12}{P_{21}}^{-1})^2, \] so the base case $n=3$ is true, as each relator lifts to either a doubled commutator transvection, a separating $\pi$-twist or the identity automorphism.

Now suppose $n > 3$. By induction, the group $\ptor_{n-1}$ is generated by the purported generating set. We lift this generating set to $\ptor_n(1)$ in the obvious way.

By Lemma \ref{jcap}, we need only add in $\mathcal{J}_n(1)$-conjugates of the words $[P_{ij},P_{i1}]$ and  $[P_{ij},P_{j1}]P_{i1}^2$, for $1 < i \neq j \leq n$. The former are clearly conjugate in $\pia_n$ to the doubled commutator transvection $[P_{12},P_{13}]$. For the latter, observe that \[ [P_{ij},P_{j1}]P_{i1}^2 = [P_{ij},P_{i1}^{-1}], \] which again is conjugate in $\pia_n$ to $[P_{12},P_{13}]$, so we are done.
\end{proof}

Theorem A allows us to conclude that $\ptor_n$ is normally generated in $\pia_n$ by the automorphisms $\chi_1 = [P_{12},P_{13}]$ and \[ \chi_2 = (P_{23}{P_{13}}^{-1}P_{31}P_{32}P_{12}{P_{21}}^{-1})^2. \] Let $\Omega_n \leq \pia_n$ denote the symmetric group on $X$. The presentation for $\Gamma_n[2] \cong \ppia_n / \ptor_n$ given in Corollary \ref{congpres} follows from Theorem A by adding the $\Omega_n$-orbits of $\chi_1$ and $\chi_2$ to Collins' presentation for $\ppia_n$ as relators, then applying the obvious Tietze transformations.

We now demonstrate that the presence of separating $\pi$-twists in our generating set for $\ptor_n$ is necessary.
\begin{prop}\label{necgen}For $n \geq 3$, the group generated by doubled commutator transvections is a proper subgroup of $\ptor_n$.\end{prop}
\begin{proof}Let $\mathcal{D}$ denote the subgroup of $\ptor_n$ generated by doubled commutator transvections. In other words, $\mathcal{D}$ is the normal closure of $\chi_1 = [P_{12},P_{13}]$ in $\pia_n$. Then the $\Omega_n$-orbit of $\chi_1$ is a normal generating set for $\mathcal{D}$ in $\ppia_n$. Adding the members of this orbit to the presentation of $\ppia_n$ as relators yields a finite presentation $\mathcal{Q}$ of $\ppia_n / \mathcal{D}$, which may be altered using Tietze transformations so that it looks like the presentation in Corollary \ref{congpres}, with relator 10 (and relator 7, if $n=3$) removed, (where we interpret $S_{ij}$ and $O_i$ as formal symbols, rather than matrices). We shall show that the relations of $\mathcal{Q}$ are not a complete set of relations on the generating set $\{ S_{ij}, O_i \}$ for $\Gamma_n[2] \cong \ppia_n / \ptor_n$, and so conclude that $\mathcal{D} \neq \ptor_n$.

It is easily shown that \[ \xi := (S_{32}{S_{31}}^{-1}S_{13}S_{23}S_{21}{S_{12}}^{-1})^2, \] the image of $\chi_2$ in $\Gamma_n[2]$, is trivial, but we shall show that $\xi$ is non-trivial in the group presented by $\mathcal{Q}$. Observe that by trivialising all the generators of $\Gamma_n[2]$ except for $S_{12}$ and $S_{21}$, we surject $\Gamma_n[2]$ onto the free Coxeter group generated by the images of $S_{12}$ and $S_{21}$, say $A$ and $B$, respectively. This is easily verified by examining the relators of $\mathcal{Q}$. The image of $\xi$ under this map is $ABAB \neq 1$, and so $\xi$ is non-trivial in the group presented by $\mathcal{Q}$. Therefore $\mathcal{D}$ is a proper subgroup of $\ptor_n$.
\end{proof}
Note that in the proof of Proposition \ref{necgen} we also showed that relators 1--9 of Corollary \ref{congpres} are not a sufficient set of relators that hold between the $O_i$ and $S_{jk}$, as relator 10 is not a consequence of the others. This allows us to conclude that the quotient space $\pbc_3 / \ptor_3$ is not simply-connected.
\begin{corollary}\label{not1conn} The complex $\pbc_3 / \ptor_3$ is not simply-connected. \end{corollary}
\begin{proof}By Theorem \ref{armstrong}, the complex $\pbc_3 / \ptor_3$ is simply-connected if and only if $\ptor_3$ is generated by the vertex stabilisers of the action of $\ptor_3$ on $\pbc_3$. As in the proof of Theorem A, the group generated by the vertex stabilisers of this action may be normally generated in $\pia_3$ by the group $\ptor_3(1)$. The same calculations as in the proof of Theorem A show that $\ptor_3(1)$ is the normal closure of the doubled commutator transvection $[P_{12}, P_{13}]$. However, Proposition \ref{necgen} showed that this normal closure is a proper subgroup of $\ptor_3$, so the quotient $\pbc_3 / \ptor_3$ is not simply-connected.\end{proof}

\section{The connectivity of $\pbc_n$ and its quotient} In this section, we determine the levels of connectivity of $\pbc_n$ and $\pbc_n / \ptor_n$. The former is found to be simply-connected, following the same approach as Day--Putman \cite{DayPutman}, while the latter is shown to be closely related to a complex already studied by Charney \cite{Charney}, which is $(n-3)$-connected.

\textbf{The connectivity of $\pbc_n$.} First, we recall the definition of the Cayley graph of a group. Let $G$ be a group with finite generating set $S$. The \emph{Cayley graph of $G$ with respect to $S$}, denoted $\mathrm{Cay}(G,S)$, is the graph with vertex set $G$ and edge set $\{(g,gs) \mid g \in G, s \in S^{\pm1} \}$, where an ordered pair $(x,y)$ indicates that vertices $x$ and $y$ are joined by an edge. If $s \in S$ has order 2, we identify each pair of edges $(g,gs)$ and $(g,gs^{-1})$ for each $g \in G$, to ensure that the Cayley graph is simplicial. Similarly, we also insist that the identity element of $G$ is excluded from $S$.

We establish Theorem \ref{pbcconn} by constructing a map $\Psi$ from the Cayley graph of $\pia_n$ to $\pbc_n$ and demonstrating that the induced map of fundamental groups is both surjective and trivial. We require the Cayley graph of $\pia_n$ with respect to a particular generating set, which we now describe. Assume that $n \geq 3$. For $1 \leq i \neq j < n$, let $t_{ij}$ permute $x_i$ and $x_j$, fixing $x_k$ with $k \neq i, j$. Using the symmetric group action on $X$, we deduce from Proposition \ref{palgens} that we may generate $\pia_n$ using the set \[ Z:= \{ t_{ij}, \iota_2, \iota_3, P_{21}, P_{23}, P_{31}, P_{34} \mid 1 \leq i \neq j \leq n \}. \]

We may use the symmetric group action on $X$ to streamline the presentation of $\pia_n$ given in Section 2, to obtain the following list of defining relators for $\pia_n$ on the generating set $Z$:
\begin{multicols}{2}
\begin{enumerate}
\item $t_{ij} = t_{ji}$,
\item ${t_{ij}}^2 = 1$,
\item $ut_{ij}u^{-1} = t_{u(i)u(j)}$,
\item ${\iota_2}^2 = 1$,
\item $(\iota_2 \iota_3)^2 = 1$,
\item $[\iota_2, P_{31}] = 1$,
\item $(\iota_2P_{21})^2 = 1$,
\item $(\iota_3P_{23})^2 = 1$,
\item $P_{23}P_{31}P_{21} = {P_{21}}^{-1} P_{31}P_{23}$,
\item $[P_{21},P_{31}] = 1$,
\item $[P_{21},P_{34}] = 1$,
\item $\iota_3 = t_{23}\iota_2 t_{23}$,
\item $P_{31} = t_{23}P_{21}t_{23}$,
\item $P_{23} = t_{13} P_{21} t_{13}$,
\item $P_{34} = t_{14}t_{23} P_{21} t_{23} t_{14}$,
\item $P_{21} = w P_{21} w^{-1}$ for $w \in \mathcal{W}$,
\item $\iota_2 = v \iota_2 v^{-1}$ for $v \in \mathcal{V}$,
\end{enumerate}
\end{multicols}
where $1 \leq i \neq j \leq n$, $u \in \{ t_{ij} \}$, and $\mathcal{W}$ and $\mathcal{V}$ are the sets of words on $\{t_{ij} \}$ that fix both $x_1$ and $x_2$, and only $x_2$, respectively. The relations of type 16 and 17 arise due to the streamlining of the presentation of $\pia_n = \epia_n \rtimes \Omega^{\pm1}(X)$ given in Section 2. Note that relations 1-3 are a complete set of relations for the symmetric group, when generated by the transpositions $\{t_{ij} \}$ \cite{Putman}.

We now consider the Cayley graph $\cay(\pia_n,Z)$. Observe that for each $z \in Z$, either $z(x_1) = x_1$ or $\{x_1, z(x_1) \}$ forms a partial $\pi$-basis for $F_n$. This allows us to construct a map of complexes from the star of the vertex 1 in $\cay(\pia_n,Z)$ to $\pbc_n$, by mapping an edge $z \in Z^{\pm1}$ to the edge $v_1 - z(v_1)$ (which may be degenerate). Using the actions of $\pia_n$ on $\cay(\pia_n,Z)$ and $\pbc_n$, we can extend this map to a map of complexes $\Psi : \cay(\pia_n,Z) \to \pbc_n$. Explicitly, $\Psi$ takes a vertex $z_1 \ldots z_r$ of $\cay(\pia_n,Z)$ ($z_i \in Z^{\pm1}$) to the vertex $z_1 \ldots z_r (x_1)$.

\begin{proof}[Proof of Theorem \ref{pbcconn}] This proof is modelled on Day--Putman's proof of their Theorem A \cite{DayPutman}. Let $$\Psi_\ast : \pi_1(\cay(\pia_n,Z),1) \to \pi_1 (\pbc_n, x_1)$$ be the map of fundamental groups induced by $\Psi$. Explicitly, the image of a loop $z_1 \ldots z_k$ ($z_i \in Z^{\pm1}$) in $\pi_1(\cay(\pia_n,Z),1)$ under $\Psi_\ast$ is \[x_1 - z_1(x_1) - z_1z_2 (x_1) - \ldots z_1z_2 \ldots - z_k(x_1) = x_1.\] We first show that $\Psi_\ast$ is the trivial map, then show that it is also surjective.

Recall that the Cayley graph $C$ of a group $G$ with presentation $\langle X \mid R \rangle$ forms the 1-skeleton of its \emph{Cayley complex}, which we obtain by attaching disks along the loops in $C$ corresponding to all conjugates in $G$ of the words in $R$. It is well-known that the Cayley complex of a group $G$ is simply-connected \cite[Proposition 4.2]{LyndonSchupp}. We now verify that the loops in $\pi_1(\cay(\pia_n,Z),1)$ corresponding to the relators in above the streamlined presentation for $\pia_n$ have trivial image under $\Psi_\ast$. This allows us to extend $\Psi$ to a map from the (simply-connected) Cayley complex of $\pia_n$ (rel. $Z$), and so conclude that $\Psi_\ast$ is trivial.

Note that in the following we confuse a relator with the loop in $\pi_1(\cay(\pia_n,Z),1)$ to which it corresponds. Many of the relators 1--17 map to $x_1$ in $\pbc_n$, as they are words on members of $Z$ that fix $x_1$. The only ones we need to check are 1-3 and 14-17. Relators 1--3 map into the contractible simplex spanned by $x_1, \ldots, x_n$, so are trivial. Relators 14 and 15 are mapped into the simplices $x_1 - x_3$ and $x_1 - x_4$, respectively. We rewrite relators 16 and 17 as $P_{21}w = wP_{21}$ and $\iota_2 v = v \iota_2$. It is clear, then, that relators of type 16 map into the contractible subcomplex of $\pbc_n$ spanned by $x_1, \ldots, x_n$ and $x_1x_2x_1$, and relators of type 17 map into the contractible subcomplex spanned by $x_1, {x_2}^{\pm1}, \ldots, x_n$. All relators have now been dealt with, so we conclude that $\Psi_\ast$ is the trivial map.

We argue as in Day--Putman's proof \cite{DayPutman} for the surjectivity of $\Psi_\ast$. We represent a loop $\omega \in \pi_1 (\pbc_n, x_1)$ as \[x_1 = w_0 - w_1 - \ldots - w_k = x_1,\] for some $k \geq 0$. We will demonstrate that for any such path (not necessarily with $w_k = x_1$), there exist $\phi_1, \ldots, \phi_k \in \pia_n(1)$ such that \[ w_i = \phi_1 t_{12} \phi_2 t_{12} \ldots \phi_i t_{12} (x_1),\] for $0 \leq i \leq k$. We use induction. In the case $k=0$, there is nothing to prove. Now suppose $k > 0$. Consider the subpath \[ w_0 - w_1 - \ldots - w_{k-1}. \] By induction, to prove the claim all we need find is $\phi_k \in \pia_n(1)$ such that \[w_k = \phi_1 t_{12} \ldots \phi_k t_{12}(x_1). \] We know that $w_{k-1} = \phi_1 t_{12} \ldots \phi_{k-1} t_{12}(x_1)$ and $w_k$ form a partial $\pi$-basis, therefore so do $x_1$ and $(\phi_1 t_{12} \ldots \phi_{k-1} t_{12})^{-1}(w_k)$. By construction, the action of $\pia_n$ is transitive on the set of two-element partial $\pi$-bases, so there exists $\phi_k \in \pia_n(1)$ mapping $x_2$ to $(\phi_1 t_{12} \ldots \phi_{k-1} t_{12})^{-1}(w_k)$. Therefore \[w_k = \phi_1 t_{12} \ldots \phi_k t_{12}(x_1), \] as required.

Now, we define \[\phi_{k+1} = (\phi_1 t_{12} \ldots \phi_k t_{12})^{-1},\] so that \[R:= \phi_1 t_{12} \ldots \phi_k t_{12} \phi_{k+1} = 1\] is a relation in $\pia_n$. Observe that since $w_k = x_1$, we have $\phi_{k+1} \in \pia_n(1)$. Also, the generating set $Z$ contains a subset that generates $\pia_n(1)$, by Proposition \ref{palgens}. We are thus able to write \[\phi_i = z^i_1 \ldots z^i_{p_i} ,\] for some $z^i_{j} \in Z^{\pm1}$ ($1 \leq i \leq k+1$, $1 \leq j \leq p_i$), each of which fixes $x_1$. We see that $R \in \pi_1 (\mathrm{Cay}(\pia_n, Z), 1)$ maps to $\omega \in \pi_1 (\pbc_n,x_1)$. Removing repeated vertices, $R$ maps to \[ x_1 - \phi_1t_{12}(x_1) - \ldots - \phi_1t_{12}\ldots \phi_kt_{12}(x_1) = x_1, \] which equals $\omega$ by construction. Hence $\Psi_\ast$ is surjective as well as trivial, so $\pi_1 (\pbc_n, x_1) = 1$.
\end{proof}

\textbf{The connectivity of $\pbc_n / \ptor_n$.} A complex analogous to $\pbc_n$ may be defined when working over $\Z^n$ rather than $F_n$. We write $\mathcal{B}_n(\Z)$ for the \emph{complex of partial bases of $\Z^n$}, whose $(k-1)$-simplices correspond to subsets $\{u_1, \ldots, u_k\}$ of free abelian bases of $\Z^n$. Writing members of $\Z^n$ multiplicatively, there is an analogous notion of an odd palindrome on some fixed free abelian basis $V$, and so also of a partial $\pi$-basis. The \emph{complex of partial $\pi$-bases of $\Z^n$} is defined in the obvious way, and denoted $\pbc_n(\Z)$. Just as $\pia_n$ acts transitively on the set of $\pi$-bases of $F_n$, so does $\Gamma_n[2]$ act transitively on the set of $\pi$-bases of $\Z^n$, as we now verify.

\begin{lemma} \label{transact} The group $\Gamma_n[2]$ acts transitively on the set of $\pi$-bases of $\Z^n$.
\end{lemma}
\begin{proof}By definition, any $\pi$-basis is of the form $\{ Mv_1, \dots, Mv_n \}$, for $M \in \Gamma_n[2]$ and $\{v_1, \dots , v_n \}$ is the standard basis of $\Z^n$, where $v_i$ has 1 in the $i$th positions and 0s elsewhere. Thus, we have a well-defined action of $\Gamma_n[2]$ on the set of $\pi$-bases of $\Z^n$ by left-multiplication of basis elements, which is transitive, as every $\pi$-basis lies in the same orbit as $\{v_1, \dots, v_n \}$.
\end{proof}

We first show that $\pbc_n / \ptor_n \cong \pbc_n(\Z)$, then show that $\pbc_n(\Z)$ is $(n-3)$-connected using a related complex studied by Charney. To prove the former, the following lemma is required.
\begin{lemma}\label{bascomp}Fix $\{u_1,\ldots,u_n\}$ as a $\pi$-basis for $\Z^n$, and let $\rho : F_n \to \Z^n$ be the abelianisation map. Let $\tilde U = \{\tilde u_1, \ldots, \tilde u_k\}$ be a partial $\pi$-basis of $F_n$ such that $\rho(\tilde u_i) = u_i$ for each $1 \leq i \leq k$. Then we can extend $\tilde U$ to a $\pi$-basis of $F_n$, $\{\tilde u_1, \ldots, \tilde u_n\}$, such that $\rho(\tilde u_i) = u_i$ for $1 \leq i \leq n$.
\end{lemma}
\begin{proof}Extend $\{\tilde u_1, \ldots, \tilde u_k\}$ to a full $\pi$-basis of $F_n$, $\{\tilde u_1, \ldots, \tilde u_k, \tilde u_{k+1}', \ldots, \tilde u_n'\}$, and define $u_j' = \rho(\tilde u_j')$ for $k+1 \leq j \leq n$. Then $\{u_1, \ldots, u_k, u_{k+1}', \ldots, u_n' \}$ is a $\pi$-basis for $\Z^n$. By Lemma~\ref{transact}, the group $\Gamma_n[2]$ acts transitively on the set of $\pi$-bases of $\Z^n$, so there exists $\phi \in \Gamma_n[2](k)$ such that $\phi(u_j') = u_j$ for $k+1 \leq j \leq n$. By Proposition \ref{commdiag}, $\phi$ lifts to some $\tilde \phi \in \ppia_n(k)$, and the $\pi$-basis $\{\tilde u_1, \ldots, \tilde u_k, \tilde \phi( \tilde u_{k+1}'), \ldots, \tilde \phi (\tilde u_n ')\}$ projects onto $\{u_1, \ldots, u_n\}$ as desired.
\end{proof}
Now we establish an isomorphism of simplicial complexes between $\pbc_n / \ptor_n$ and $\pbc_n(\Z)$.
\begin{theorem}The spaces $\pbc_n / \ptor_n$ and $\pbc_n(\Z)$ are isomorphic as simplicial complexes.\end{theorem}
\begin{proof}Let $\rho: F_n \to \Z^n$ be the abelianisation map, and define a map of simplicial complexes $\Phi : \pbc_n \to \pbc_n(\Z)$ on simplices by $\{w_1, \ldots, w_k \} \mapsto \{\rho(w_1), \ldots, \rho(w_k)\}$, for $1 \leq k \leq n$. The map $\Phi$ is surjective: by Lemma \ref{bascomp}, each $\pi$-basis of $\Z^n$ is projected onto by some $\pi$-basis of $F_n$, and $\pi$-bases of $\Z^n$ correspond to maximal simplices of $\pbc_n(\Z)$.

It is clear that the map $\Phi$ is invariant under the action of $\ptor_n$ on $\pbc_n$, and so $\Phi$ factors through $\pbc_n / \ptor_n$. To establish the theorem, all we need do is show that the induced map from $\pbc_n / \ptor_n \to \pbc_n(\Z)$ is injective. In other words, we must show that if two simplices $s, s'$ of $\pbc_n$ have the same image under $\Phi$, then $s$ and $s'$ differ by the action of some member of $\ptor_n$.

Suppose that $s = \{w_1, \ldots, w_k\}$ and $s' = \{w_1', \ldots, w_k'\}$ have the same image under $\Phi$. We may assume that $\rho(w_i) = \rho(w_i')$ for $1 \leq i \leq k$. Let $\Phi(s) = \{\bar w_1, \ldots, \bar w_k\}$, and extend this partial $\pi$-basis of $\Z^n$ to a full $\pi$-basis, $W = \{\bar w_1, \ldots, \bar w_n \}$. By Lemma \ref{bascomp}, we may extend $\{w_1, \ldots, w_k\}$ to $\{w_1, \ldots, w_n\}$ and $\{w_1', \ldots, w_k'\}$ to $\{w_1', \ldots, w_n'\}$, such that both of these full $\pi$-bases map onto $W$. Define $\theta \in \pia_n$ by $\theta(w_i) = w_i'$ for $1 \leq i \leq n$. By construction, $\theta(s) = s'$ and $\theta \in \ptor_n$, so the theorem is proved.
\end{proof}
This more explicit description of $\pbc_n / \ptor_n$ as $\pbc_n(\Z)$ enables us to investigate the quotient's connectivity.

\begin{proof}[Proof of Theorem \ref{quoconn}]By a \emph{unimodular sequence} in $\Z^n$, we mean an (ordered) sequence $(u_1, \ldots, u_k) \subset (\Z^n)^k$ whose entries form a basis of a direct summand of $\Z^n$. Observe that this is just an ordered version of the notion of a partial basis of $\Z^n$. The set of all such sequences of length at least one form a poset under subsequence inclusion. Charney considers (among others) the subposet of sequences $(u_1, \ldots, u_k)$ such that each $u_i$ is congruent to a standard basis vector $v_j$ under mod 2 reduction of the entries of $u_i$. We denote by $\mathcal{X}_n$ the poset complex given by the subposet of such sequences. Theorem 2.5 of Charney says that $\mathcal{X}_n$ is $(n-3)$-connected.

Let $\pbc_n(\Z)^\ast$ denote the barycentric subdivision of $\pbc_n(\Z)$. Label each vertex of $\pbc_n(\Z)^\ast$ by the partial $\pi$-basis associated to the simplex of $\pbc_n(\Z)$ to which the vertex corresponds. Define a simplicial map $h : \mathcal{X}_n \to \pbc_n(\Z)^\ast$ by $(u_1, \ldots, u_k) \mapsto \{u_1, \ldots, u_k\}$. We may think of $h$ as `forgetting the order' of each unimodular sequence. Comparing the definitions of $\mathcal{X}_n$ and $\pbc_n(\Z)$, it is not immediately clear that $h$ is well-defined, as there might be some vertex $(u_1,\ldots,u_k)$ of $\mathcal{X}_n$ such that $\{u_1, \ldots, u_k \}$ extends to a full basis of $\Z^n$, but not a full $\pi$-basis. However, viewing the full basis of $\Z^n$ as a matrix in $\Gamma_n[2]$, a straightforward column operations argument shows that this cannot be the case, so $h$ is well-defined.

We see that $h$ induces a map $\pi_i (\mathcal{X}_n) \to \pi_i (\pbc_n(\Z)^\ast)$ for $i \geq 0$, and show that the induced map is surjective. Set a consistent lexicographical order on the vertices of $\pbc_n(\Z)^\ast$, and view $\omega \in \pi_i (\pbc_n(\Z)^\ast)$ as a simplicial $i$-sphere. The chosen lexicographical ordering allows us to lift $\omega$ to $\pi_i (\mathcal{X}_n)$, so the induced maps are surjective. The statement of the theorem follows immediately, since $\pi_i(\mathcal{X}_n) = 1$ for $0 \leq i \leq n-3$.
\end{proof}

\section{A presentation for $\Gamma_3[2]$} In order to apply Armstrong's theorem \cite{Armstrongorbit}, it must be the case that $\pbc_n / \ptor_n \cong \pbc_n(\Z)$ is simply-connected. However, as we have seen from Corollary \ref{not1conn}, the space $\pbc_3(\Z)$ has non-trivial fundamental group. The case $n=3$ forms the base case of our inductive proof of Theorem \ref{mainthm}, so we require an alternative approach to find a generating set for $\ptor_3$. Our approach is to find a specific finite presentation of $\Gamma_3[2]$, and use the short exact sequence
\[ 1 \longrightarrow \ptor_3 \longrightarrow \ppia_3 \longrightarrow \Gamma_3[2] \longrightarrow 1 \] to lift the relators in the presentation of $\Gamma_3[2]$ to a normal generating set for $\ptor_3$.

\textbf{The augmented partial $\pi$-basis complex for $\Z^3$.} By adding simplices to the complex $\pbc_3(\Z)$, we obtain a simply-connected complex that $\Gamma_3[2]$ acts on. This action allows us to present $\Gamma_3[2]$.

Recall that $\mathcal{B}_n(\Z)$ is the partial basis complex of $\Z^n$. We represent its vertices by column vectors $u = \begin{pmatrix} u^{(1)} \\ \vdots \\u^{(n)} \end{pmatrix}$. For use in the proof of Theorem \ref{augment}, we follow Day--Putman \cite{DayPutman} and define the \emph{rank of $u$} to be $|u^{(n)}|$, and denote it by $R(u)$. Let $\mathcal{Y}$ denote the full subcomplex of $\mathcal{B}_3(\Z)$ spanned by $\pbc_3(\Z)$ and vertices $u$ for which $u^{(1)}$ and $u^{(2)}$ are odd and $u^{(3)}$ is even. We call $\mathcal{Y}$ the \emph{augmented partial $\pi$-basis complex for $\Z^3$}. We now demonstrate that $\mathcal{Y}$ is simply-connected.
\begin{theorem}\label{augment}The complex $\mathcal{Y}$ is simply-connected. \end{theorem}
\begin{proof}By Theorem 2.5 of Charney \cite{Charney}, we know that $\pbc_3(\Z)$ is 0-connected, and hence so is $\mathcal{Y}$. To show that $\mathcal{Y}$ is simply-connected, we adapt the proof of Theorem B of Day--Putman \cite{DayPutman}.

Let $u$ be a vertex of a simplicial complex $C$. The \emph{link of $u$ in $C$}, denoted $\lk_C(u)$, is the full subcomplex of $C$ spanned by vertices joined by an edge to $u$. Let $v_3 \in \Z^3$ be the standard basis vector with third entry 1 and 0s elsewhere. Observe that for any vertex $u \in \mathcal{Y}$, we have $\lk_\mathcal{Y} (u) \cong \lk_\mathcal{Y}(v_3)$. This is because the group generated by $\Gamma_3[2]$ and the matrix \[ E = \begin{bmatrix}  1 & 0 & 0 \\ 1 & 1 & 0 \\ 0 & 0 & 1 \end{bmatrix}\] acts simplicially on $\mathcal{Y}$ and transitively on the $0$-skeleton of $\mathcal{Y}$. This action is transitive on vertices because $\Gamma_3[2]$ acts transitively on the vertices of $\pbc_3(\Z)$, and any vertex of $\mathcal{Y} \setminus \pbc_3(\Z)$ may be taken to a vertex of $\pbc_3(\Z)$ by acting on it with $E$.

We begin by establishing that $\lk_\mathcal{Y} (v_3)$ is connected (and hence, by the above, so is the link of any vertex of $\mathcal{Y}$). By considering what the columns of $M \in \mathrm{GL}(3,\Z)$ whose final column is $v_3$ must look like, we see that a necessary and sufficient condition for $\begin{pmatrix} u^{(1)} \\ u^{(2)} \\u^{(3)} \end{pmatrix}$ to be a member of $\lk_\mathcal{Y}(v_3)$ is that $\begin{pmatrix} u^{(1)} \\ u^{(2)} \end{pmatrix}$ is a vertex of $\mathcal{B}_2(\Z)$. The link $\lk_\mathcal{Y} (v_3)$ may thus be described as follows: it has one vertex for each pair $(a, b)$, where $a$ is a vertex of $\mathcal{B}_2(\Z)$ and $b \in 2 \Z$, with vertices $(a, b)$ and $(c,d)$ joined by an edge if and only if $a$ and $c$ are joined by an edge in $\mathcal{B}_2(\Z)$. Hence $\lk_\mathcal{Y}(v_3)$ is connected, though note that its fundamental group is an infinite rank free group.

Now, let $\omega \in \pi_1( \mathcal{Y}, v_3)$. We represent $\omega$ by the sequence of vertices \[ w_0 - w_1 - \ldots - w_r, \] where $w_i$ ($1 \leq i \leq r$) are vertices of $\mathcal{Y}$, and $w_0 = w_r = v_3$. Our goal is to systematically homotope this loop so that the rank of each vertex in the sequence is 0. Such a loop may be contracted to the vertex $v_3$, and so is trivial in $\pi_1(\mathcal{Y})$.

Consider a vertex $w_i$ for some $1 < i < r$, with $R(w_i) \neq 0$. Since $\lk_\mathcal{Y}(w_i)$ is connected, there is some path \[w_{i-1} - q_1 - q_2 - \ldots - q_s - w_{i+1}\] in $\lk_\mathcal{Y}(w_i)$, as seen in Figure \ref{hmtpy}.
 \begin{figure}
\centering
\def\svgwidth{4.5in}
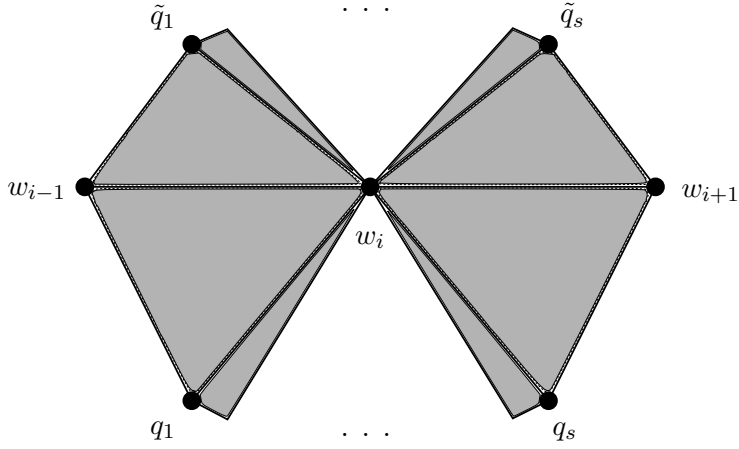
\caption{We find two homotopic paths that bound a disk inside $\lk_{\mathcal{Y}}(w_i)$, where the `upper' path seen here is constructed so that $R(\tilde q_j) < R(q_j)$ for $1 \leq j \leq s$.}
\label{hmtpy}
\end{figure}Fix attention on some $q_j$ ($1 \leq j \leq s$). By the Division Algorithm, there exists $a_j, b_j \in \Z$ such that $R(q_j) = a_j \cdot R(w_i) + b_j$ such that $0 \leq b_j < R(w_i)$. As in the proof of Lemma \ref{conggens}, we wish to ensure that $a_j$ is even, if possible. In all but one case, we will be able to rewrite the Division Algorithm as $R (q_j) = A_j \cdot R(w_i) + B_j$, for some $A_j, B_j \in \Z$ such that $A_j$ is even and $0 \leq |B_j| < R(w_i)$. We do a case-by-case parity analysis. Note that since $q_j$ and $w_i$ are joined by an edge, $R(q_j)$ and $R(w_i)$ cannot both be odd, otherwise $q_j$ and $w_i$ would both map to the same member of $(\Z / 2)^3$ when we reduce their entries mod 2. This would prohibit $\{q_j, w_i \}$ from extending to a basis $J$ of $\Z^3$, otherwise the image of $J$ in $(\Z / 2)^3$ would generate despite only having 2 members. If $R(q_j)$ and $R(w_i)$ have different parities and $a_j$ is odd, we may take $A_j = a_j + 1$ and $B_j = b_j - R(w_i)$. In that case, $|B_j| < R(w_i)$, since $b_j$ must be odd and hence non-zero. If both $R(q_j)$ and $R(w_i)$ are even, we may still do this, unless $b_j = 0$.

We now associate to each $q_j$ a new vertex, $\tilde q_j$, defined by
\[ \tilde q_j = \begin{cases} q_j - a_j \cdot w_i & \phantom{\phi \phi \phi} \mbox{ if } a_j \mbox{ even,}\\
 q_j - A_j \cdot w_i & \phantom{\phi \phi \phi} \mbox{ if } a_j \mbox{ odd, } b_j \neq 0,
\\ q_j - a_j \cdot w_i & \phantom{\phi \phi \phi} \mbox{ if } a_j \mbox{ odd, } b_j = 0 \end{cases}. \] Note that when $b_j = 0$, $R(\tilde q_j) = 0$, and under the conditions given, $\tilde q_j$ is always well-defined as a vertex of $\mathcal{Y}$. The path
\[w_{i-1} - q_1 - \ldots - q_s - w_{i+1}\] is homotopic inside $\lk_\mathcal{Y}(w_i)$ to the path
\[w_{i-1} - \tilde q_1 - \ldots - \tilde q_s - w_{i+1} ,\] as seen in Figure \ref{hmtpy}. By construction, $R(\tilde q_j) < R(w_i)$. Iterating this procedure continually homotopes $\omega$ until it is inside the contractible (full) subcomplex spanned by $v_3$ and $\lk_\mathcal{Y}(v_3)$, and hence is trivial. Therefore $\pi_1(\mathcal{Y}) = 1$.
\end{proof}
\textbf{The complex $\pbc_3(\Z)$ is not simply-connected.} It may be tempting to try to use the method in the above proof to show that $\pbc_3(\Z)$ is simply-connected, however we know by Corollary \ref{not1conn} that $\pbc_3(\Z)$ has non-trivial fundamental group. The obstruction to the above proof going through occurs when defining $\tilde q_j$ in the case that $a_j$ is odd and $b_j = 0$, as $\tilde q_j \not \in \pbc_3 (\Z)$. When $a_j$ is odd and $b_j = 0$, there is no even multiple of $w_i$ that can be added to $q_j$ to decrease its rank, so this method of homotoping loops to a point will not work.

\textbf{Presenting $\Gamma_3[2]$.} Let $\Gamma_3[2](w_1, \ldots, w_k)$ denote the stabiliser of the ordered tuple \linebreak $(w_1, \ldots, w_k)$ of vertices of $\mathcal{Y}$. Having demonstrated that $\mathcal{Y}$ is simply-connected, we now turn our attention to the obvious action of $\Gamma_3[2]$ on $\mathcal{Y}$. This action is simplicial, does not invert edges, and the quotient complex under the action is contractible, as seen in Figure \ref{quofig}. The quotient lifts to a subcomplex $W$ of $\mathcal{Y}$ via the vertex labels seen in Figure \ref{quofig}. This subcomplex is what Brown \cite{Brownpres} refers to as a \emph{fundamental domain} for the action, and so a theorem of Brown \cite[Theorem 3]{Brownpres} allows us to conclude that $\Gamma_3[2]$ is the free product of the stabilisers of the vertices of $W$ modulo \emph{edge relations}, which identify the copies of the edge stabiliser $\Gamma_3[2](a,b)$ inside the vertex stabilisers $\Gamma_3[2](a)$ and $\Gamma_3[2](b)$, where $a, b \in \{v_1,v_2,v_3,v_1+v_2 \}$ are distinct.
\begin{figure}
\centering
\def\svgwidth{3.2in}
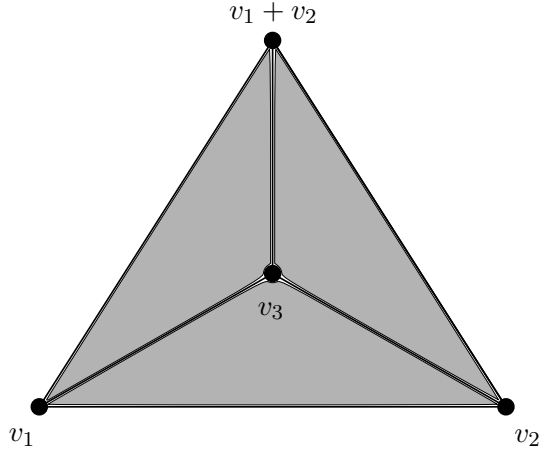
\caption{The quotient complex of $\mathcal{Y}$ under the action of $\Gamma_3[2]$. We have labelled its vertices using representatives from the vertex set of $\mathcal{Y}$.}
\label{quofig}
\end{figure}

 We obtain a finite presentation for $\Gamma_3[2](v_1)$ using the semi-direct production decomposition of $\Gamma_3[2](1)$ given by Lemma \ref{commdiag} (noting that $\Gamma_2[2] \cong \ppia_2$). The group $\Gamma_3[2](v_1)$ is generated by the set $\{O_2,O_3, S_{23}, S_{32}, S_{12}, S_{13} \}$, with a complete list of relators given by all relators of the form 1--9 (excluding 7, as it is not defined when $n=3$) seen in Corollary \ref{congpres}. By permuting the indices accordingly, we also obtain finite presentations for the stabiliser groups $\Gamma_3[2](v_2)$ and $\Gamma_3[2](v_3)$. Identifying the edge stabiliser subgroups of these three groups appropriately, we obtain the presentation seen in Corollary \ref{congpres} without relators 7 and 10; we denote this presentation by $\mathcal{P}$.
 
 We now see that the effect of identifying the edge stabiliser subgroups of $\Gamma_n[2](v_1+v_2)$ with the corresponding copies inside the other three vertex stabiliser groups is to include one additional relator: relator 10. Since $\Gamma_3[2](v_1+v_2)$ and $\Gamma_3[2](v_1)$ are conjugate inside $\mathrm{GL}(3,\Z)$, we take a formal presentation for $\Gamma_3[2](v_1+v_2)$ by adding a `hat' to each of the symbols in the presentation of $\Gamma_3[2](v_1)$

The members of $\Gamma_3[2](v_1+v_2)$ are not, however, strings of formal symbols, but are members of $\Gamma_3[2]$. To express them as such, we observe that \[ \Gamma_3[2](v_1+v_2) = E_{21} \cdot \Gamma_3[2](v_1) \cdot {E_{21}}^{-1}, \] where $E_{21}$ is the elementary matrix with 1 in the $(2,1)$ position. In Table \ref{e1conj}, we see the conjugates of the generators of $\Gamma_3[2](v_1)$ by $E_{21}$. These give expressions for the formal symbols generating $\Gamma_3[2](v_1+v_2)$. For example, \[\hat S_{12} = E_{21} S_{12} {E_{21}}^{-1} = O_1O_2 S_{21}{S_{12}}^{-1}.\]
\begin{table}[h!!!!!!]
\centering
\begin{tabular}{|c|c|}\hline
Generator $M$ of $\Gamma_3[2](v_1)$ & The conjugate $\hat M = E_{21} \cdot M \cdot {E_{21}}^{-1}$ \\
\hline $O_2$ & $S_{21}O_2$ \\
$O_3$ & $O_3$ \\
$S_{12}$ & $O_1O_2 S_{21}{S_{12}}^{-1}$ \\
$S_{13}$ & $S_{13}S_{23}$ \\
$S_{23}$ & $S_{23}$ \\
 $S_{32}$ & $S_{32}{S_{31}}^{-1}$ \\
 \hline
\end{tabular}
\caption{The conjugates of the generating set of $\Gamma_3[2](v_1)$ by $E_{21}$.}
\label{e1conj}
\end{table}

Let $f_i$ be the edge joining $v_1 +v_2$ to $v_i$ ($1 \leq i \leq 3$), and let $J_i$ be the stabiliser of $f_i$. We consider these each in turn. Observe that \[J_2 = E_{21} \cdot \Gamma_3[2](v_1,v_2) \cdot {E_{21}}^{-1}, \] so $J_2$ is generated by $\{O_3, S_{13}S_{23}, S_{23} \}$. We have expressed those three generators in terms of the generators of $\Gamma_3[2](v_1)$. To obtain the relations corresponding to this edge stabiliser, we must express them using the generators of $\Gamma_3[2](v_1+v_2)$, and set them to be equal accordingly. Consulting Table \ref{e1conj}, we get the edge relations $\hat O_3 = O_3$, $\hat S_{13} = S_{13}S_{23}$ and $\hat S_{23} = S_{23}$. Note that these relations simply reiterate the expressions we had already determined for $\hat O_3$, $\hat S_{13}$ and $\hat S_{23}$. Similarly, as we obtain $J_3$ by conjugating $\Gamma_3[2](v_1,v_3)$ by $E_{21}$, the edge relations arising from the edge $f_3$ are $\hat O_2 = S_{21}O_2$, $\hat S_{12} = O_1O_2S_{21}{S_{12}}^{-1}$ and $\hat S_{32} = S_{32}{S_{31}}^{-1}$.

Finally, to obtain $J_1$, we conjugate $\Gamma_3[2](v_1,v_2)$ by the elementary matrix $E_{12}$. We obtain that $J_1$ is generated by $\{O_3, S_{13}, S_{13}S_{23} \}$, which gives edge relations $\hat O_3 = O_3$, $S_{13} = \hat S_{13} {\hat S_{23}}^{-1}$ and $\hat S_{13} = S_{13}S_{23}$. Note that these relations all arise as consequences of the edge relations coming from the edges $f_2$ and $f_3$, so are not required.

We now use the above edge relations to replace the formal relators defining $\Gamma_3[2](v_1+v_2)$ with words on the generating set $\{S_{ij}, O_k \}$. Using Tietze transformations and Brown's Theorem 3 \cite{Brownpres}, we may then conclude that a complete presentation for $\Gamma_3[2]$ is obtained by adding these relators to the presentation $\mathcal{P}$. For example, the relator $\hat{O_2}^2$ becomes $(S_{21}O_2)^2$. All but one of these additional relators are consequences of ones already in $\mathcal{P}$. The one relator that is not is $[\hat S_{13}, \hat S_{32}] \hat{S_{12} }^{-2}$, which becomes
\[ [S_{13}S_{23}, S_{32}{S_{31}}^{-1}] (O_1O_2 S_{21}{S_{12}}^{-1})^{-2}. \] Using the other relations in $\Gamma_3[2]$, this word may be rewritten in the form of relator 10 in Corollary \ref{congpres}; we have thus verified that the presentation given in Corollary \ref{congpres} is correct when $n=3$.
\bibliography{palbib}

\end{document}